\documentclass[10pt,a4paper]{article}
\usepackage[final]{hyperref} 
\usepackage[abbrev, backrefs]{amsrefs}
\usepackage{a4wide,fancyhdr}
\usepackage[toc,page]{appendix}

\usepackage{latexsym}
\usepackage[all]{xy}
\usepackage{amsfonts}
\usepackage{amsmath}
\usepackage{amsthm}
\usepackage{amssymb}
\usepackage{graphicx}
\usepackage{pifont}
\usepackage{enumerate}
\xyoption{2cell}

\usepackage{tikz}
\usetikzlibrary{matrix,arrows}
\usepackage{amscd}

\numberwithin{equation}{section}
\usepackage{enumitem}
\newlist{abbrv}{itemize}{1}
\setlist[abbrv,1]{label=,labelwidth=1in,align=parleft,itemsep=0.1\baselineskip,leftmargin=!}

\newcounter{zlist}

\newcounter{blist}

\newcounter{rlist}

\def\stac#1{\raise-.2cm\hbox{$\stackrel{\displaystyle\otimes}{\scriptscriptstyle{#1}}$}}

\def\cten#1{\raise-.2cm\hbox{$\stackrel{\displaystyle\widehat{\otimes}}
		{\scriptscriptstyle{#1}}$}}

\def\Label#1{\label{#1}\ifmmode\llap{[#1] }\else
	\marginpar{\smash{\hbox{\tiny [#1]}}}\fi}
\def\Label{\label}

\newcounter{c}

\newcommand{\etyk}[1]{\vspace{-7.4mm}$$\begin{equation}\Label{#1}
	\addtocounter{c}{1}}
\renewcommand{\]}{\ifnum \value{c}=1 $$\else \end{equation}\fi}
\newtheorem{proposition}{Proposition}[section]
\newtheorem{lemma}[proposition]{Lemma}
\newtheorem{corollary}[proposition]{Corollary}
\newtheorem{theorem}[proposition]{Theorem}

\theoremstyle{definition}

\theoremstyle{remark}
\newtheorem{remark}[proposition]{Remark}

\usepackage{float}

\newcommand{\R}{{\mathbb R}}
\newcommand{\eps}{{\varepsilon}}
\let\ACMmaketitle=\maketitle
\renewcommand{\maketitle}{\begingroup\let\footnote=\thanks \ACMmaketitle\endgroup}

\usepackage{authblk}

\title{Groundstate asymptotics for a class of singularly perturbed $p$-Laplacian problems in $\R^N$}
\author[1]{\small Wedad Albalawi\footnote{wsalbalawi@pnu.edu.sa}}
\author[2]{\small Carlo Mercuri \footnote{C.Mercuri@swansea.ac.uk}}
\author[2]{\small  Vitaly Moroz\footnote{V.Moroz@swansea.ac.uk}}
\affil[1]{Princess Nourah bint Abdulrahman University, College of Sciences, PO Box 84428, Riyadh, KSA}
	\affil[2]{Department of Mathematics, Computational Foundry, Swansea University, Fabian Way, Swansea, SA1~8EN, UK}
\date{}

\begin{document}

\maketitle

\begin{abstract}
We study the asymptotic behavior of positive groundstate solutions to the quasilinear elliptic equation
\begin{equation} \tag{$P_\eps$} \label{0e1}
-\Delta_{p} u + \varepsilon u^{p-1} - u^{q-1}  +u^{\mathit{l}-1}  = 0 \qquad  \text{in} \quad  \mathbb{R}^{N}
\end{equation}
where $1<p<N $, $p<q<l<+\infty$ and $\varepsilon> 0 $ is a small parameter.
For $\eps\to 0$, we give a characterisation of asymptotic regimes as a function of the parameters $q$, $l$ and $N$.
In particular, we show that the behavior of the groundstates is sensitive to whether $q$ is less than, equal to, or greater than the critical Sobolev exponent $p^{*} :=\frac{pN}{N-p}$.
 
\end{abstract}

\tableofcontents

\section{Introduction}
The present paper is devoted to the study of positive solutions to the quasilinear elliptic equation
\begin{equation}\tag{$P_\eps$}\label{1s}
-\Delta_{p} u + \varepsilon u^{p-1} -  u^{q-1} + u^{l-1}  = 0 \qquad  \text{in} \quad  \mathbb{R}^{N},
\end{equation}
where $$\Delta_{p} u=\textrm{div}(|\nabla u|^{p-2} \nabla u),$$
is the $p$-Laplacian operator, $1<p<N$, $ p<q<l $ and $ \varepsilon > 0$ is a small parameter.
Our main aim is to understand the behaviour of positive groundstate solutions to \eqref{1s} as $ \varepsilon\to 0$.

\noindent By a solution to \eqref{1s} we mean a weak solution $u_{\varepsilon} \in W^{1,p}(\mathbb{R}^{N}) \cap L^{l}(\mathbb{R}^{N}) $.
These solutions are constructed as critical points of the energy
\begin{equation}\tag{$\mathcal{E_{\varepsilon}}$}\label{1e}
\mathcal{E_{\varepsilon}}(u):= \frac{1}{p} \int_{\mathbb{R^{N}}} |\nabla u|^{p} dx - \int_{\mathbb{R^{N}}} F_{\varepsilon}(u) dx,
\end{equation}	
where $$ F_{\varepsilon}(u)= \int_{0}^{u} \tilde{f}_{\varepsilon}(s) ds,$$
and the expression $\tilde{f}_\eps$ is a suitable bounded truncation of
\begin{equation}\label{feps}
f_{\varepsilon}(s):= - \varepsilon |s|^{p-2}s+ |s|^{q-2} s - |s|^{l-2} s.
\end{equation}
Throughout the paper by groundstate solution to \eqref{1s} we mean a positive weak solution which has the least energy $\mathcal{E_{\varepsilon}}$ amongst all the other non-trivial solutions.

\noindent In the first part of the paper,  for all $1<p<N$ and $ p<q<l $,  we prove the existence of a radial groundstate solution $u_\eps$ of \eqref{1s} for all sufficiently small $ \varepsilon > 0, $ see Theorem \ref{th3.2.1}, extending classical results of Berestycki and Lions \cite{Berestycki-Lions} from the  Laplacian $(p=2)$ to the $ p$-Laplacian setting, for any $ 1<p<N$.
As a byproduct of the method \cite{Berestycki-Lions} which is adapted to the present quasilinear context, the weak solutions to \eqref{1s} which are found, are essentially bounded and decay uniformly to zero as $|x|\rightarrow\infty$. We recall that, as in the known case $p=2$ treated in \cite{Berestycki-Lions}, the symmetry of the solutions is achieved as a limit of a suitable (minimising) sequence of radially decreasing rearrangements constructed from a possibly non-radial minimising sequence. Theorem \ref{th3.2.1} in Section \ref{3.1} summarises all the above results about the existence and basic properties of these groundstates to \eqref{1s}.

\noindent We  point out that for large $\eps>0$ equation \eqref{1s} has no finite energy solutions, so the restriction on the size of $\eps$ is essential for the existence of the groundstates.
The uniqueness (up to translations) of a spherically symmetric groundstate of \eqref{1s} is rather delicate.
For $p\leq 2$, Serrin and Tang \cite{Tang}*{Theorem 4} proved that equation \eqref{1s} admits at most one positive groundstate solution.
For $p>2$ the uniqueness could be also expected but to the best of our knowledge this remains an open question.
We do not study the question of uniqueness in this paper and none of our result rely on the information about the uniqueness of the groundstate to \eqref{1s}.

\noindent The question of understanding the asymptotic behaviour of the groundstates $ u_{\varepsilon}$ of  \eqref{1s} as $\varepsilon\rightarrow 0,$
naturally arises in the study of various bifurcation problems, for which \eqref{1s} at least in the case $p=2$ can be considered as a canonical normal form (see e.g. \cite{Cross}, \cite{van}). This problem may also be regarded as a bifurcation problem for quasilinear elliptic equations
$$-\Delta_{p} u = f_\eps(u)\qquad  \text{in} \quad  \mathbb{R}^{N},$$
whose nonlinearity $f_\eps$ has the leading term in the expansion around zero which coincides with the ones in \eqref{1s}.
Let us also mention that problem \eqref{1s} in the case $p=2$ appears in the study of phase transitions \cites{Cahn,Muratov,Unger}, as well as in the study of the decay of false vacuum in quantum field theories \cite{Coleman}.

\noindent Loosely speaking, to understand the asymptotic behaviour  of the groundstates $ u_{\varepsilon}$ as $\varepsilon\rightarrow 0 $, one notes that elliptic regularity implies that locally the solution $ u_{\varepsilon}$ converges as $\varepsilon\rightarrow 0 $ to a radial solution of the limit equation (see Theorem \ref{246})
\begin{equation}\tag{$P_{0}$}\label{2s}
-\Delta_{p} u - u^{q-1} + u^{l-1}= 0 \qquad  \text{in} \quad  \mathbb{R}^{N}.
\end{equation}
It is known that (here and in the rest of the paper $p^{*}:=\frac{pN}{N-p}$ is the critical Sobolev exponent):
when $q\leq p^{*} $ equation \eqref{2s} has no non-trivial finite energy solution, by  Poho\v{z}aev's identity \eqref{e3.17};
whereas for $q>p^{*}$ equation \eqref{2s} admits a radial groundstate solution. Existence goes back to Berestycki-Lions \cite{Berestycki-Lions } and Merle-Peletier \cite{Merle} in the case $p=2$ and, in the context of the present paper, it is proved in the general $p$-Laplacian case (see Theorem \ref{11224411}); uniqueness questions have been studied by Tang \cite{Tang Moxun}*{Theorem $ 4.1$}, see also Remark \ref{Tangremark} .\\

\noindent
In Theorem \ref{ss122} we prove using direct variational arguments that, as expected, for $q>p^{*}$ solutions $ u_{\varepsilon}$ converge as $\varepsilon\rightarrow 0 $ to a non-trivial radial groundstate solution to the formal limit equation \eqref{2s}.
The fact that for $q\leq p^{*}$ equation \eqref{2s} has no non-trivial positive solutions, suggests that for $q\leq p^{*} $ the solutions $ u_{\varepsilon}$ should converge almost everywhere, as $\varepsilon\rightarrow 0 ,$ to the trivial zero solution
of equation \eqref{2s} (see estimate \eqref{esubueps}). This however does not reveal any information about the limiting profile of  $ u_{\varepsilon}$.
Therefore, instead of looking at the formally obtained limit equation \eqref{2s},  we are going to show that for $q\leq p^{*} $ solutions $ u_{\varepsilon}$ converge to a non-trivial limit {\em after a rescaling}. The limiting profile of  $ u_{\varepsilon}$ will be obtained from the groundstate solutions of the {\em limit equations} associated with the rescaled equation \eqref{1s}, where the choice of the associated rescaling and limit equation depends on the value of $p$ and on the space dimension $N$ in a highly non-trivial way.

\noindent The convergence of rescaled solutions $ u_{\varepsilon}$  to their limiting profiles will be proved using a variational analysis similar to the techniques developed in \cite{Moroz-Muratov} in the case $p=2$.
Note that the natural energy space for equation \eqref{1s} is the usual Sobolev space
\begin{equation*}
W^{1,p}(\mathbb{R}^{N}):= \Big\{u: u \in L^{p}(\mathbb{R}^{N}) \quad  \text {and} \quad   \nabla u \in L^{p}(\mathbb{R}^{N})\Big \},
\end{equation*}
with the norm $$||u||_{1,p} = ||u||_{p}+||\nabla u||_{p} , $$
while for $q> p^{*} $ the natural functional setting associated with the limit equation \eqref{2s} is the homogeneous Sobolev space $D^{1,p}(\mathbb{R}^{N}) $ defined for $1<p<N$ as the completion of $C^{\infty}_{0}(\mathbb{R}^{N}) $ with respect to the norm $||\nabla u||_{L^{p}} $. Since $W^{1,p}(\mathbb{R}^{N})\varsubsetneq D^{1,p}(\mathbb{R}^{N}),$ it follows that  no natural perturbation setting (suitable to apply the implicit function theorem or Lyapunov-Schmidt type reduction methods) is available to analyse the family of equations \eqref{1s} as $\varepsilon\rightarrow 0 $.
In fact, even for $p=2$ a linearisation of \eqref{2s} around the groundstate solution is not a Fredholm operator and has zero as the bottom of the essential spectrum in $ L^{2}(\mathbb{R}^{N})$. 
In the case of the $p$--Laplace equations the difficulty in applying classical perturbation methods is even more striking, as for $1<p<2$ the energy associated with the $p$-Laplacian is not twice Fr\'echet differentiable.\\
In order to understand the limiting profile of  $ u_{\varepsilon}$ in the case  $q\leq p^{*} $,
we introduce the {\it canonical} rescaling associated with the lowest order nonlinear term in \eqref{1s}:
\begin{equation}\label{03s}
v_{\varepsilon}(x)=\varepsilon^{-\frac{1}{q-p}}u_{\varepsilon}\big(\frac{x}{\sqrt[p]{\varepsilon}}\big).
\end{equation}
Then \eqref{1s} reads as 
\begin{equation}\tag{$R_{\varepsilon}$}\label{3s}
-\Delta_{p} v+ v^{p-1}  = v^{q-1}  -\varepsilon^{\frac{l-q}{q-p}} v^{l-1}   \qquad  \text{in} \quad  \mathbb{R}^{N},
\end{equation}
from which we formally get, as  $\varepsilon\rightarrow 0,$ the limit problem
\begin{equation}\tag{$R_{0}$}\label{4s}
-\Delta_{p} v+ v^{p-1}  = v^{q-1}     \qquad  \text{in} \quad  \mathbb{R}^{N}.
\end{equation}
We recall that for $q \geq p^{*} $ equation \eqref{4s} has no non-trivial finite energy solutions, as a consequence of Poho\v{z}aev's identity \eqref{e3.17}; whereas
for $ p<q< p^{*} $ equation \eqref{4s} possesses a unique radial groundstate solution. Existence was proved by Gazzola, Serrin and Tang \cite{Gazzola-Serrin-Tang} and uniqueness by Pucci-Serrin \cite{Pucci-Serrin}*{Theorem $2$}.
\noindent The particular rescaling \eqref{03s} allows to have, when $p<q< p^{*} ,$ for both \eqref{3s} and the limit problem \eqref{4s}, a variational formulation on the same Sobolev space $W^{1,p}(\mathbb{R}^{N}) $. This indicates that problem \eqref{3s} could be considered as a small perturbation of the limit problem \eqref{4s}. In particular in the case $p=2$ the family of the groundstates $(v_{\varepsilon}) $ of problem \eqref{3s} could be rigorously interpreted as a perturbation of the groundstate solution of the limit problem \eqref{4s} using the perturbation techniques and framework developed by Ambrosetti, Malchiodi et al., see \cite{Ambrosetti} and references. However, for $p\neq 2$ the Lyapunov--Schmidt reduction technique, in the spirit of \cite{Ambrosetti} is not directly applicable. Instead, in this work, using a direct variational argument inspired by \cite{Moroz-Muratov}*{Theorem $2.1$} we prove (see Theorem \ref{s121}) that for $p<q<p^*$ groundstate solutions $(v_{\varepsilon}) $ of the rescaled problem \eqref{3s} converge to the (unique) radial groundstate of the limit problem \eqref{4s}. \\

\noindent In the critical case $q=p^*$, the limit problem \eqref{4s} has no non-trivial positive solutions. This means that in this case the {\it canonical} rescaling \eqref{03s} does not accurately capture the behaviour of $(u_{\varepsilon}) $. In the present paper, extending the results obtained in \cite{Moroz-Muratov} for $p=2$, we show that for $q=p^{*}$ the asymptotic behaviour of the groundstate solutions to \eqref{1s} after a rescaling is given by a particular solution of the critical Emden-Fowler equation
\begin{equation}\tag{$R_{*}$} \label{5s1}
- \Delta_{p} U = U^{p^{*}-1} \quad \text{in} \; \mathbb{R}^{N}.
\end{equation}
It is well-known that equation \eqref{5s1} admits a continuum of radial groundstate solutions.
We will prove that the choice of the rescaling (and a particular solution of \eqref{5s1}) which provides the limit asymptotic profile for groundstate solutions to equation \eqref{1s} depends on the dimension $N$ in a non-trivial way (see Theorem \ref{s123}).

Wrapping up, we provide a characterisation of the three asymptotic regimes occurring as $\varepsilon\rightarrow0 $, i.e. the subcritical case $q<p^{*} $, the supercritical case $q>p^{*} $ and the critical case $q=p^{*},$ extending the results of \cite{Muratov} and \cite{Moroz-Muratov}, to both a singular ($p<2$) and degenerate ($p>2$) quasilinear setting.

\subsection*{Asymptotic notation}

Throughout the paper we will extensively use the following asymptotic notation. For $\eps \ll 1$ and $f(\eps), g(\eps) \geq 0$, we write $f(\eps)
\lesssim g(\eps)$, $f(\eps) \sim g(\eps)$ and $f(\eps) \simeq
g(\eps)$, implying that there exists $\eps_0 > 0$ such that for every
$0 < \eps \leq \eps_0$: \smallskip

$f(\eps)\lesssim g(\eps)$ if there exists $C>0$ independent of $\eps$
such that $f(\eps) \le C g(\eps)$;

$f(\eps)\sim g(\eps)$ if $f(\eps)\lesssim g(\eps)$ and
$g(\eps)\lesssim f(\eps)$;

$f(\eps)\simeq g(\eps)$ if $f(\eps) \sim g(\eps)$ and $\lim_{\eps\to
	0}\frac{f(\eps)}{g(\eps)}=1$.

\noindent
We also use the standard Landau symbols $f = O(g)$ and $f = o(g)$, with the understanding that $f \geq 0$ and $g \geq 0$.
As usual, $C,c,c_1$, etc., denote generic positive constants
independent of $\eps$.

\section{Main results}

The following theorem summarizes the existence results for the equation \eqref{1s}.
The proof is a standard adaptation of the Berestycki and Lions method \cite{Berestycki-Lions}.
For completeness, we sketch the arguments in Section \ref{3.1}.

\begin{theorem}\label{th3.2.1}
Let $N\ge 2$, $1<p<N$ and $p<q<l$. Then there exists $\eps_*=\eps_*(p,q,l)>0$ such that
 for all $\eps\in(0,\eps_*)$, equation \eqref{1s} admits a groundstate $u_{\varepsilon}\in W^{1,p}(\mathbb{R}^{N}) \cap L^{l}(\mathbb{R}^{N})\cap C^{1,\alpha}_{loc}(\mathbb{R}^{N})$. Moreover, $u_\eps(x)$ is, by construction, a positive monotone decreasing function of $|x|$ and \begin{equation*}
 u_{\varepsilon}(|x|) \leq C e^{-\delta|x|}, \qquad  x \in \mathbb{R}^{N},
 \end{equation*}
 for some $ C, \delta > 0 $. 
\end{theorem}

For $p\leq 2$, Serrin and Tang proved \cite{Tang}*{Theorem 4} that equation \eqref{1s} admits at most one positive groundstate solution.
For $p>2$ the uniqueness to the best of our knowledge remains an open question.
As anticipated earlier, none of our subsequent results rely on the uniqueness of groundstates of $(P_\eps)$.
In what follows, $u_\eps$ always denotes {\it a} groundstate solution to \eqref{1s} constructed in Theorem \ref{th3.2.1} for an $\eps\in(0,\eps_*)$. When we say that groundstates $u_\eps$ converge to a certain limit (in some topology) as $\eps\to 0$, we understand that for every $\eps>0$ a groundstate of \eqref{1s} is selected, so that $(u_\eps)_{\eps\in(0,\eps_*)}$ is a branch of groundstates of \eqref{1s}, which is not necessarily continuous in $\eps$. In the present work we study the limit behaviour of such a branch of groundstates when $\eps\to 0$.

\subsection{Subcritical case $p<q< p^{*}$}

As anticipated earlier, since in the subcritical case the formal limit equation \eqref{2s} has no groundstate solutions, the family of groundstates $u_{\varepsilon}$ must converge to zero, uniformly on compact subsets. We describe the asymptotic behaviour of $ u_{\varepsilon}$ performing the rescaling \eqref{03s} which transforms \eqref{1s} into equation \eqref{3s}.
In Section \ref{5}, using the variational approach developed in the main part of this work we prove the following result, which extends \cite{Moroz-Muratov}*{Theorem $2.1$} to the case $p\neq 2$.

\begin{theorem}\label{s121}
Let $N\ge 2$, $1<p<N$, $p<q<p^{*}$ and $(u_{\varepsilon})$ be a family of groundstates of $(P_\eps)$.
As $\varepsilon \rightarrow 0 $, the rescaled family 
\begin{equation}\label{e-canon-resc}
v_{\varepsilon}(x):=\varepsilon^{-\frac{1}{q-p}}u_{\varepsilon}\Big(\frac{x}{\sqrt[p]{\varepsilon}}\Big)
\end{equation}
converges in $ W^{1,p}(\mathbb{R}^{N})$, $L^{l}(\mathbb{R}^{N})$ and $ C_{loc}^{1,\alpha}(\mathbb{R}^{N})$
to the unique radial groundstate solution $v_{0}(x) $ of the limit equation \eqref{4s}.
In particular,
\begin{equation}\label{esubueps}
u_{\varepsilon}(0)\simeq\varepsilon^{\frac{1}{q-p}} v_{0}(0).
\end{equation}
\end{theorem}

\subsection{Critical case $q=p^{*}$}
In this case we show that after a suitable rescaling the correct limit equation for \eqref{1s} is given by the critical Emden-Fowler equation
\begin{equation}\tag{$R_{*}$}\label{5s}
- \Delta_{p} U = U^{p^{*}-1} \quad \text{in} \; \mathbb{R}^{N}.
\end{equation}
It is well-known by Guedda-Veron \cite{Guedda} that the only radial solution to \eqref{5s} is given, up to the sign, by
the family of rescalings
\begin{equation} \label{9s}
U_{\lambda}(|x|):= U_1(|x|/\lambda)\qquad(\lambda>0),
\end{equation}
where
\begin{equation}\label{111s}
U_{1}(|x|):= \left[\frac{\kappa^{1/p'}N^{1/p}}{1+|x|^{p'}}\right]^{\kappa/p'},
\end{equation}
and where $ p':=\frac{p}{p-1} $ and $ \kappa:=\frac{N-p}{p-1} $. Recently in \cite{FaMeWi} it has been observed that $\pm U_{\lambda}$ are the only nontrivial radial solutions to $\Delta_pu+|u|^{p^*-2}u=0$ in $D^{1,p}(\R^N)$. Sciunzi \cite{Sciunzi} and V\'etois \cite{Vet}, respectively in the ranges $p>2$ and $p<2,$ proved that any positive solution to \eqref{5s} in $D^{1,p}(\R^N)$ is necessarily radial about some point; this combined with \cite{Guedda} gives a complete classification of the positive finite energy solutions to \eqref{5s}. 

Our main result in this work is the following theorem, which extends \cite{Moroz-Muratov}*{Theorem $2.5$} to the case $p\neq 2$.

\begin{theorem}\label{s123}
Let $N\ge 2$, $1<p<N$, $p^{*}=q<l$ and $(u_{\varepsilon})$ be a family of groundstates of $(P_\eps)$. 
There exists a rescaling
\begin{equation}
\lambda_{\varepsilon}:(0,\varepsilon_{*})\rightarrow(0,\infty)
\end{equation}
such that as $\varepsilon\rightarrow 0 $, the rescaled family 
$$ v_{\varepsilon}(x):= \lambda_{\varepsilon}^{\frac{N-p}{p}} u_{\varepsilon}(\lambda_{\varepsilon}x)$$
converges in $D^{1,p}(\mathbb{R}^{N})$
to the radial groundstate solution $U_{1}(x)$ of the Emden--Fowler equation \eqref{5s}.
Moreover,
\begin{equation}\label{22s}
\lambda_{\varepsilon}\gtrsim \begin{cases}
\varepsilon^{-\frac{p^{*}-p}{p(l-p)}} &  \;1<p<\sqrt{N},\\
\Big(\varepsilon (\log \frac{1}{\varepsilon})\Big)^{-\frac{(p^{*}-p)}{p(l-p)}} & p=\sqrt{N},\\
\varepsilon^{-\frac{1} {[(l-p^{*})(p-1)+p]}} & \sqrt{N}<p<N,
\end{cases}
\end{equation}
and
\begin{equation}\label{3.23+}
\lambda_{\varepsilon}\lesssim \begin{cases}
\varepsilon^{-\frac{p^{*}-p}{p(l-p)}} &  \;1<p<\sqrt{N},\\
\varepsilon^{-\frac{(p^{*}-p)}{p(l-p)}}\Big( \log \frac{1}{\varepsilon}\Big)^{\frac{(l-p^{*})}{p(l-p)}} & p=\sqrt{N},\\
\varepsilon^{-\frac{(p^{2}-N)(l-p^{*})+p^{2}} {p^{2}[(l-p^{*})(p-1)+p]}} & \sqrt{N}<p<N.
\end{cases}
\end{equation}
\end{theorem}

\begin{remark}
The lower bound \eqref{22s} on $\lambda_\eps$ can be converted into an upper bound on the maximum of $u_\eps$,
\begin{equation}\label{222s}
u_{\varepsilon}(0)\lesssim
\begin{cases}
\varepsilon^{\frac{l}{(l-p)}}& 1<p<\sqrt{N},\\
(\varepsilon (\log \frac{1}{\varepsilon})\Big)^{\frac{l}{(l-p)}} & p=\sqrt{N},\\
\varepsilon^{\frac{N-p} {p[(l-p^{*})(p-1)+p]}} & \sqrt{N}<p<N,
\end{cases}
\end{equation}
see Corollary \ref{3121}.
\end{remark}

For $1<p<\sqrt{N}$ lower bound \eqref{22s} and upper bound \eqref{3.23+} are equivalent and hence optimal.
For $\sqrt{N}\le p<N$, the upper bounds in \eqref{3.23+} do not match the lower bounds \eqref{22s}.
However, under some additional restrictions we could obtain optimal two--sided estimates.

\begin{theorem}\label{main-cor}
Under the assumptions of Theorem \ref{s123}, we additionally have 
\begin{equation}\label{2255s}
\lambda_{\varepsilon}\sim \begin{cases}
\varepsilon^{-\frac{p^{*}-p}{p(l-p)}} &  \text{$1<p<\sqrt{N}$ and $N\ge 2$},\\
\Big(\varepsilon (\log \frac{1}{\varepsilon})\Big)^{-\frac{(p^{*}-p)}{p(l-p)}} & \text{$p=\sqrt{N}$ and $N\geq4$},\\
\varepsilon^{-\frac{1} {[(l-p^{*})(p-1)+p]}} &\text{$\sqrt{N}<p<\frac{N+1}{2}$ and $N\geq4$},
\end{cases}
\end{equation}
and
\begin{equation}\label{2255ss}
u_{\varepsilon}(0)\sim
\begin{cases}
\varepsilon^{\frac{l}{(l-p)}}& \text{$1<p<\sqrt{N}$ and $N\ge 2$},\\
(\varepsilon (\log \frac{1}{\varepsilon})\Big)^{\frac{l}{(l-p)}} & \text{$p=\sqrt{N}$ and $N\geq4$},\\
\varepsilon^{\frac{N-p} {p[(l-p^{*})(p-1)+p]}} &\text{$\sqrt{N}<p<\frac{N+1}{2}$ and $N\geq4$}.
\end{cases}
\end{equation}
In the above cases $v_\eps$ converges to $U_1(x)$ in $L^{l}(\mathbb{R}^{N})$ and $C^{1,\alpha} _{loc}(\mathbb{R}^{N}).$
\end{theorem}

\begin{remark}
In the case $p=2$ and $N\ge 3,$  {\em two--sided} asymptotics of the form \eqref{2255s} were derived in \cite{Muratov} using methods of formal asymptotic expansions. Later, two sided bounds of the form \eqref{2255s} were rigorously established for $p=2$ in \cite{Moroz-Muratov}*{Theorem $2.5$}. The barrier approach developed in \cite{Moroz-Muratov}*{Lemma 4.8} in order to refine upper bounds on $\lambda_\eps$ in the difficult case $\sqrt{N}\leq p<N$ cannot be fully extended to $p\neq 2$, see Lemma \ref{3.1.9}.
In this difficult case, the matching upper bounds of the form \eqref{22s} are valid for $\sqrt{N}\leq p<\frac{N+1}{2}$ and $N\geq4$.
\end{remark}

\begin{remark}
Theorem \ref{main-cor} leaves open the following cases, where matching lower and upper bounds are not available:
\begin{itemize}
\item $N\ge 4$ and $\frac{N+1}{2}\le p <N$
\item $N=3$ and $\sqrt{3}\le p<3$
\item $N=2$ and $\sqrt{2}\le p<2$
\end{itemize}
Note that the case $N=3$ and $p=2$ is not included in Theorem \ref{main-cor}. However, matching bound  
\eqref{2255s} and \eqref{2255ss} remain valid in this case. This is one of the results in \cite{Moroz-Muratov}*{Theorem $2.5$}.
We conjecture that the restriction $p<\frac{N+1}{2}$ is merely technical and is due to the method we use.
\end{remark}

\subsection{Supercritical case $q>p^{*}$}
Unlike the subcritical and critical cases, for $ q>p^{*}$ the formal  limit equation \eqref{2s} admits a nontrivial solution. 
Using a direct analysis of the family of constrained minimization problems associated with \eqref{1s}, we prove the following result, which extends \cite{Moroz-Muratov}*{Theorem $ 2.3$} to the case $p\neq 2$.

\begin{theorem}\label{ss122}
Let $N\ge 2$, $1<p<N$, $p^{*}<q<l$ and $(u_{\varepsilon})$ be a family of groundstates of $(P_\eps)$. As $\varepsilon \rightarrow0$, the family  $u_{\varepsilon} $ converges in $D^{1,p}(\mathbb{R}^N)$, $L^{l}(\mathbb{R}^N)$ and $C^{1,\alpha}_{loc} (\mathbb{R}^{N})$
to a groundstate solution $u_{0}(x) $ of the limit equation \eqref{2s}, with
 \begin{equation*}
u_{0}(x)\sim |x|^{-\frac{N-p}{p-1}} \quad \text{as}\quad |x| \rightarrow \infty.
\end{equation*}
Moreover, it holds that
$$ u_{\varepsilon}(0)\simeq u_{0}(0),$$
and that $\varepsilon||u_{\varepsilon}|_{p}^{p}\rightarrow0. $
\end{theorem}

\subsection{Organisation of the paper}
This paper is organised as follows. Section \ref{BLsection} is devoted to the existence and qualitative properties of groundstates $ u_{\varepsilon}$ to \eqref{1s}; in Section \ref{PDEsection} we deal with existence and qualitative properties of groundstates to the limiting PDEs \eqref{2s}, \eqref{4s}, \eqref{5s1}. Both sections contain various facts about the equation $(P_\eps)$ and limiting equations which are involved in our analysis. In the rest of the paper we study the asymptotic behavior of the groundstates $ u_{\varepsilon}$.
In {Section \ref{6}} we study the most delicate critical case $q=p^{*} $ and prove Theorem \ref{s123} and Theorem \ref{main-cor}. In {Section \ref{7}} we consider the supercritical case $q>p^{*} $ and prove Theorem \ref{ss122}. In {Section \ref{5}} we consider the subcritical case  $q<p^{*}$ and prove Theorem \ref{s121}. For the reader convenience we have collected in the sections A and B of the Appendix some auxiliary results which have been used in the main body of the paper.

\section{Groundstate solutions to (\ref{1s})}\label{BLsection}

\subsection[Necessary conditions and Poho\v{z}aev's identity]{Necessary conditions and Poho\v{z}aev's identity}\label{10.1}
According to Poho\v{z}aev's  classical identity \cite{Pohozaev} for $p$-Laplacian equations, a solution to \eqref{1s} which is smooth enough, necessarily satisfies the identity
\begin{equation}\label{e3.17}
\int\limits_{\mathbb{R}^{N}} |\nabla u|^{p} dx = p^{*} \int\limits_{\mathbb{R}^{N}} F(u) dx,
\end{equation}
for $1< p< N  $.
Identities of this type are classical, see for instance \cite{Pucci} for $C^2$ solutions and \cite{Degiovanni} for bounded domains. In the present paper the following version of Poho\v{z}aev's identity has been extensively used.
\begin{proposition}\label{2.2.12}
Suppose $ f:\mathbb{R} \rightarrow \mathbb{R} $ is a continuous function such that $ f(0) = 0 $, and set $ F(t) = \int_{0}^{t} f(s) ds  $. Let \begin{equation*}
u \in C_{loc}^{1,\alpha}(\mathbb{R}^{N}), \quad \textrm {and} \quad |\nabla u|^p, \,\, F(u)\in L^{1}(\mathbb{R}^{N})
\end{equation*} with $u$ such that
\begin{equation*}
-\Delta_{p} u = f(u)\quad
\end{equation*}
holds in the sense of distributions. Then $ u $ satisfies \eqref{e3.17}.
\end{proposition}

\begin{proof}

We first assume that $p\leq 2.$ By the classical regularity result of Tolksdorf \cite{Tol}, see also Theorem 2.5 in \cite{PucciC},
we have
\begin{equation*}
u\in W^{2,p}_{loc}(\mathbb R^N),\quad p\leq2.
\end{equation*}
Having checked the existence and local summability of the second weak derivatives in this case we argue as follows. Multiply the equation by $x_i\partial_i u(x)$ and integrate over $B_R=B(0,R)$  and denote by $n(\cdot)$ the outer normal unit vector. Observe that the vector field $$v= x_i\partial_i u  |\nabla u|^{p-2}\nabla u$$
is such that  $v\in C(\R^N,\R^N)$ and $\textrm{div}\, v \in L^1_{loc}(\R^N).$
By the divergence theorem (see e.g. Lemma 2.1 in \cite{Mercuri}) we have $$\int_{B_R} \Delta_p u \,x_i\partial_i u(x)dx=\int_{\partial B_R}|\nabla u(\sigma)|^{p-2} \partial_i u (\sigma)\sigma_i \nabla u \cdot n \,d\sigma$$$$-\int_{B_R} |\nabla u(x)|^{p-2}\nabla u(x)\cdot\nabla[x_i\partial_i u(x)]dx.$$
Write the last integral as $A_{i}+B_{i},$ where $$A_{i}:=\int_{B_R} |\nabla u(x)|^{p-2}|\partial_i u(x)|^2 dx,$$ $$B_{i}:=\frac{1}{p}\int_{B_R}\partial_i (|\nabla u(x)|^{p})x_idx.$$  An integration by parts in $B_{i}$ yields $$B_{i}= \frac{1}{p}\int_{\partial B_R}|\nabla u(\sigma)|^{p} \sigma_i n_i d\sigma-\frac{1}{p}\int_{B_R} |\nabla u(x)|^{p}dx.$$ On the other hand we have also $$\int_{B_R} f(u(x))x_i\partial_i u(x)dx$$ $$=-\int_{B_R} F(u(x))dx+\int_{\partial B_R} F(u(x))\sigma_i n_i d\sigma.$$ Summing up on $i$ we have  $$(*) \qquad N\int_{B_R}  F(u(x))dx+\Big(1-\frac{N}{p}\Big)\int_{B_R} |\nabla u(x)|^{p}dx=\int_{\partial B_R} |\nabla u(\sigma)|^{p-2}\nabla u\cdot \sigma\,\nabla u\cdot n  \,d\sigma$$ $$-\frac{1}{p}\int_{\partial B_R} |\nabla u(\sigma)|^{p}\sigma\cdot n  \,d\sigma+\int_{\partial B_R} F(u(x))\sigma \cdot n \, d\sigma.$$ The right hand side is bounded by $$M(R)=\Big(1+
\frac{1}{p}\Big)R \int_{\partial B_R} |\nabla u(\sigma)|^{p}  \,d\sigma + R\int_{\partial B_R} |F(u(x))|d\sigma.$$ Similarly as in Lemma 2.3 from \cite{Mercuri}, since $F(u), |\nabla u|^p \in L^1(\R^N)$ there exists a sequence $R_k \rightarrow \infty$ such that $M(R_k)\rightarrow 0.$ By using the monotone convergence theorem in $(*)$ we obtain the conclusion in the case $p\leq 2.$ \newline For $p> 2$ a regularisation argument similar to \cite{DiBenedetto}*{p. 833} (see also \cite{Guedda2, MercuriB, FaMeWi})
allows to work with a $C_{loc}^{1,\alpha}$-approximation $u_\varepsilon \in C^2$ which classically solves \begin{align*}
-\textrm{div}\left(\left(\varepsilon+ |\nabla u_\varepsilon|^2\right)^{\frac{p-2}{2}}\nabla u_\varepsilon\right)= f(u) \quad \quad &\textrm{in}\,\, B_{2R}, \\
u_\varepsilon = u \quad \quad  & \textrm{on} \,\,  \partial B_{2R}.
\end{align*}

The proof can be then carried out with obvious modifications of the proof given in the case $p\leq 2,$  performing the $\varepsilon$-limit before letting $R \rightarrow +\infty$ along a suitable sequence $(R_k)_{k\in \mathbb N},$ and this concludes the proof.
\end{proof}

\subsection[Variational characterisation of the solutions]{Existence and variational characterisation of the groundstates}\label{3.1}

To prove the existence of groundstates, we first observe that the method of Berestycki-Lions \cite{Berestycki-Lions} although focused on the case $p=2$ is applicable in the present quasilinear context, we sketch the proof referring to \cite{Berestycki-Lions} for the details. In fact, observe that 
$f_{\varepsilon}(s)=|s|^{q-2}s-|s|^{l-2}s -\varepsilon |s|^{p-2}s$ satisfies
\setlist{nolistsep}
\begin{enumerate}
\item[$ (f_{1}) $]
$ -\infty < \displaystyle\liminf_{s \to 0^{+}} \frac{f_{\varepsilon}(s)}{s^{p-1}} \leqq \displaystyle\limsup_{s \to 0^{+}}\frac{f_{\varepsilon}(s)}{s^{p-1}}= -\varepsilon < 0 $.
\item[$ (f_{2}) $]
$ -\infty \leqq \displaystyle\limsup_{s\to +\infty} \frac{f_{\varepsilon}(s)}{s^{p^{*}-1}} \leqq 0 $, \qquad  where \quad    $ p^{*}= \frac{pN}{N-p} $.
\item[$ (f_{3}) $]
There exists $\varepsilon_*>0$ such that for all $\varepsilon\in(0,\varepsilon_*)$ the following property holds: there exists $ \zeta > 0 $ such that $ F_{\varepsilon}(\zeta)= \int_{0}^{\zeta} f_{\varepsilon}(s) ds > 0 $.
\end{enumerate}
To prove the existence of an optimiser, one carries on with the constrained minimisation argument as in \cite{Berestycki-Lions}, based on the truncation of the nonlinearity $f_\varepsilon,$ which allows to use $W^{1,p}(\R^N)$ for the functional setting. For all $\varepsilon\in(0,\varepsilon_*)$ in the present context $p\neq 2$ a suitable truncated function $ \tilde{f_{\varepsilon}}:\mathbb{R} \rightarrow \mathbb{R} $ is provided by:

\begin{equation}\label{p21}
\tilde{f_{\varepsilon}}(u)=\begin{cases}
0, &u<0, \\
u^{q-1}-u^{l-1} -\varepsilon u^{p-1}, & u \in [0,1], \qquad \tilde F_{\varepsilon}(u):=\int\limits_{0}^{u} \tilde f_{\varepsilon}(s)ds, \\
-\varepsilon,& u>1. \\
\end{cases}
\end{equation}
Replacing in \eqref{1s} the nonlinearity with the above bounded truncation $\tilde{f}_{\varepsilon}(u)$ makes the minimisation problem 
\begin{equation}\tag{$S_{\varepsilon}$}\label{e2.4}
S_{\varepsilon}=\inf\Big\{ \int\limits_{\mathbb{R}^{N}}|\nabla w|^{p} dx;\quad  w\in W^{1,p}(\mathbb{R}^{N}), \quad  p^{*}\int\limits_{\mathbb{R}^{N}} \tilde F_{\varepsilon}(w) dx= 1\Big\}
\end{equation}
well-posed in $W^{1,p}(\mathbb{R}^{N})$ even for supercritical $l > p^{*}$. Standard compactness arguments using radially symmetric rearrangements of minimising sequences allows to obtain a radially decreasing optimiser  $w_\eps,$ see also Appendix \ref{A.3}.
If $ w_{\varepsilon} $ is an optimiser for \eqref{e2.4} then a Lagrange multiplier  $ \theta_{\varepsilon} $ exists such that
\begin{equation}\label{e3.7tilde}
-\Delta_{p} w_{\varepsilon} = \theta_{\varepsilon}  \tilde f_{\varepsilon}(w_{\varepsilon}) \quad  \text{in} \,  \mathbb{R}^{N}.
\end{equation}
Note that by construction $ \tilde{f_{\varepsilon}} (u)\in L^{\infty} (\mathbb{R}^{N})$ and then by a classical result of DiBenedetto, see e.g. Corollary p. 830 in \cite{DiBenedetto}, any solution $u\in W^{1,p}(\mathbb{R}^{N})$ to the truncated problem with $\tilde{f_{\varepsilon}}$ is regular, i.e. $u\in C_{loc}^{1,\alpha}(\mathbb{R}^{N})$. 
Then  the maximum principle implies that any solution for the truncated problem is strictly positive and solves the problem 
\begin{equation}\label{e3.7}
-\Delta_{p} w_{\varepsilon} = \theta_{\varepsilon}  f_{\varepsilon}(w_{\varepsilon}) \quad  \text{in} \,  \mathbb{R}^{N},
\end{equation}
involving the original nonlinearity. The exponential decay estimate \eqref{exp-dec} on $w_\eps$ follows by Gazzola-Serrin $($\cite{Gazzola}*{Theorem $8$}$).$ 
\noindent As a consequence of the regularity and summability, $w_{\varepsilon}$ satisfies both Nehari's identity
\begin{equation}\label{p24}
\int\limits_{\mathbb{R^{N}}}|\nabla w_{\varepsilon}|^{p} dx= \theta_{\varepsilon} \int\limits_{\mathbb{R^{N}}}f_{\varepsilon}(w_{\varepsilon}) w_{\varepsilon} dx,
\end{equation}
and Poho\v{z}aev's identity \eqref{e3.17} 
\begin{equation}\label{p25}
\int\limits_{\mathbb{R^{N}}}|\nabla w_{\varepsilon}|^{p} dx= \theta_{\varepsilon} p^{*} \int\limits_{\mathbb{R^{N}}}F_{\varepsilon}(w_{\varepsilon})dx.
\end{equation}
The latter immediately implies that
\begin{equation}\label{p26}
\theta_{\varepsilon}=S_{\varepsilon}.
\end{equation}
Then a direct calculation involving \eqref{p26} shows that the rescaled function
\begin{equation}\label{p27}
u_{\varepsilon}(x) := w_{\varepsilon}(x/\sqrt[p]{S_{\varepsilon}})
\end{equation}
is the radial groundstate of \eqref{1s}, described in Theorem \ref{th3.2.9} below.

One more consequence of Poho\v{z}aev's identity \eqref{p25} is an expression for the total energy of the solution
$$\mathcal{E_{\varepsilon}}(u_{\varepsilon})=\Big(\frac{1}{p}-\frac{1}{p^{*}}\Big) S_{\varepsilon}^{N/p},$$
(see \cite{Berestycki-Lions}* {Corollary $2$}), which shows that $u_{\varepsilon}$ is indeed a groundstate, i.e. a nontrivial solution with the least energy.  Another simple consequence of \eqref{p25} is that \eqref{1s} has no nontrivial finite energy solutions for $\varepsilon\geq \varepsilon_{*}$. The threshold value $\eps_\ast$ is simply the smallest value of $\eps > 0$ for which the energy $\mathcal{E}_\eps$ is non-negative and can be computed explicitly.

To summarize, in the spirit of \cite{Berestycki-Lions}*{Theorem $2$} we have the following

\begin{theorem}\label{th3.2.9}
Let $N\ge 2$, $1<p<N$ and $p<q<l$. Then there exists $\eps_*=\eps_*(p,q,l)>0$ such that
for all $\eps\in(0,\eps_*)$, the minimization problem \eqref{e2.4} has a minimizer $w_\eps\in  W^{1,p}(\mathbb{R}^{N}) \cap L^{l}(\mathbb{R}^{N})\cap C^{1,\alpha}_{loc}(\mathbb{R}^{N})$. 
The minimizer $w_\eps$ satisfies
\begin{equation}\label{e3.7-0}
-\Delta_{p} w_{\varepsilon} = S_{\varepsilon}  f_{\varepsilon}(w_{\varepsilon}) \quad  \text{in} \,  \mathbb{R}^{N}.
\end{equation}
Moreover, $w_\eps(x)$ is a positive monotone decreasing function of $|x|$ and \begin{equation}\label{exp-dec}
 w_{\varepsilon}(|x|) \leq C e^{-\delta|x|}, \qquad  x \in \mathbb{R}^{N},
 \end{equation}
 for some $ C, \delta > 0 $. 
The rescaled function
$$u_{\varepsilon}(x) := w_{\varepsilon}(x/\sqrt[p]{S_{\varepsilon}})$$ 
is a groundstate solution to \eqref{1s}.
\end{theorem}

In view of \eqref{p21} and since we are interested only in positive solutions
of $(P_\eps)$, in what follows we always assume that the nonlinearity $f_\eps(u)$ in $(P_\eps)$ is replaced
by its bounded truncation $\tilde{f_\eps}(u)$ from \eqref{p21}, without mentioning this explicitely.

\begin{remark}
Equivalently to \eqref{e2.4}, we can consider minimising the quotient
\begin{equation*}
\mathcal{S}_{\varepsilon}(w):= \frac{||\nabla w||_{p}^{p} }{\Big(p^{*}\int\limits_{\mathbb{R}^{N}} F_\eps(w)dx\Big)^{(N-p)/N} }, \quad w\in \mathcal{M_{\varepsilon}},
\end{equation*}
where
$$  \mathcal{M_{\varepsilon}}:= \Big\{0\leq w \in W^{1,p}(\mathbb{R}^{N}),\; \int\limits_{\mathbb{R}^{N}} F_\eps(w)dx>0\Big\}.$$
Setting $w_{\lambda} (x) := w(\lambda x),$ it is easy to check that $\mathcal{S}_{\varepsilon}(w_{\lambda})=\mathcal{S}_{\varepsilon}(w)$ for all  $\lambda > 0.$ Therefore it holds that

\begin{equation}\label{p29}
S_{\varepsilon}=\displaystyle\inf_{ w\in  \mathcal{M_{\varepsilon}}}\mathcal{S}_{\varepsilon}(w).
\end{equation}
Moreover the inclusion $ \mathcal{M}_{\varepsilon_{2}}\subset \mathcal{M}_{\varepsilon_{1}}$ for $ \varepsilon_{2} > \varepsilon_{1} > 0$, \eqref{p29} implies that $S_{\varepsilon}$ is a nondecreasing function of $\varepsilon\in  (0, \varepsilon_{*})$.
\end{remark}

\section{Limiting PDEs}\label{PDEsection}
\subsection{Critical Emden-Fowler Equation}

In this section, we recall some old and new results for the critical Emden-Fowler equation
\begin{equation}\tag{$R_{*}$}\label{e2.22a}
- \Delta_{p} u = u^{p^{*}-1},\quad u\in D^{1,p}(\mathbb{R}^{N}), \quad u> 0,
\end{equation}
where $ 1<p<N $, $ p^{*}=pN/(N-p) $  is the critical exponent for the Sobolev embedding.
We observe that any nontrivial non-negative solution to \eqref{e2.22a} is necessarily positive as a consequence of strong maximum principle (see \cite{Vazquez}). Solutions of \eqref{e2.22a} are critical points of the functional
\begin{equation}
\mathcal{J}(u):= \frac{1}{p} \int_{\mathbb{R}^{N}} |\nabla u|^{p} dx -\frac{1}{p^{*}} \int_{\mathbb{R}^{N}}|u|^{p^{*}} dx.
\end{equation}
By the Sobolev embedding $D^{1,p}(\mathbb{R}^{N})\subset L^{p^{*}}(\mathbb{R}^{N}),$ $\mathcal{J}$ is defined in $D^{1,p}(\R^N).$ Since by \cite{Lion} all the minimising sequences for
\begin{equation}\tag{$S_{*}$}\label{2p}
S_{*}:=\inf\Big\{\int\limits_{\mathbb{R}^{N}}|\nabla w|^{p} dx ; \quad w\in D^{1,p}(\mathbb{R}^{N}),\quad \int\limits_{\mathbb{R}^{N}} |w|^{p^{*}} dx =1\Big \},
\end{equation}
are relatively compact modulo translations and dilations, critical points for $\mathcal{J}$ are provided by direct minimisation, after suitable rescaling of positive solutions $W$ to the Euler--Lagrange equation for $S_*$
\begin{equation}\label{e2.25a}
-\Delta_{p} W =\theta W^{p^{*}-1} \quad \textrm{in} \; \mathbb{R}^{N}.
\end{equation}
Here since
$$\int\limits_{\mathbb{R}^{N}}|\nabla W|^{p}dx =\theta \int\limits_{\mathbb{R}^{N}}|W|^{p^{*}} dx=\theta,$$
it follows that $S_{*}=\theta .$ 
Positive finite energy solutions to this equation are classified after the works of Guedda-Veron \cite{Guedda} and of Sciunzi \cite{Sciunzi} and V\'etois \cite{Vet} mentioned in the introduction, which we recall in the following 
\begin{theorem} \label{2.3.2} 
Let $ 1<p<N $. Then every radial solution $U$ to \eqref{e2.22a} is represented as
\begin{equation}\label{eq2.27}
U(|x|)=U_{\lambda,0}(|x|):= \left[\frac{\lambda^{p'/p}k^{1/p'}N^{1/p}}{\lambda^{p'}+|x|^{p'}}\right]^{k/p'},
\end{equation}
for some $ \lambda >0 $, where $ p':=\frac{p}{p-1} $ and $ k:=\frac{N-p}{p-1} ,$ \cite{Guedda}. \newline
In fact, every solution $U$ to \eqref{e2.22a} is radially symmetric about some points $y\in\R^N$
and therefore it holds that
\begin{equation}\label{eq2.28}
U(x)=U_{\lambda,y}(x):= \left[\frac{\lambda^{p'/p}k^{1/p'}N^{1/p}}{\lambda^{p'}+|x-y|^{p'}}\right]^{k/p'},
\end{equation}
for some $ \lambda >0 $ and $ y\in \mathbb{R}^{N}, $ \cite{Sciunzi,Vet}.
\end{theorem}
In the case $ p=2 $ and $ N\geq 3 $ this result is classical, see \cite{Caffarelli}.
Hence, the radial groundstate of \eqref{e2.22a} is given by rescaling the function
\begin{equation}\label{232}
U_{1,0}(x):= \left[\frac{k^{1/p'}N^{1/p}}{1+|x|^{p'}}\right]^{k/p'},
\end{equation}
and moreover it holds that
\begin{equation}\label{e-normU}
||\nabla U_{\lambda,0}||^{p}_{p}=||U_{\lambda,0}||^{p^{*}}_{p^{*}}= S^{N/{p}}_{*},
\end{equation}
see e.g. \cite{Talenti}.
In conclusion all the positive minimizers for \eqref{2p} are translations of the radial family
\begin{equation}\label{234}
W_{\lambda}(x) := U_{\lambda,0}(\sqrt[p]{S_{*}}x).
\end{equation}

\subsection{Supercritical zero mass equation}\label{20004}
This section is devoted to the supercritical equation
\begin{equation}\tag{$P_{0}$}\label{32}
-\Delta_p u-|u|^{q-2}u+|u|^{l-2}u=0\quad\text{in } \;\mathbb{R}^N,
\end{equation}
where $1<p<N$ and $p^*<q<l$.

\begin{remark}\label{320}
Note that by Poho\v{z}aev's  identity \eqref{e3.17}, equation \eqref{32} has no solution in $D^{1,p}(\R^N)\cap C^{1,\alpha}_{loc}(\R^N)$ $q\leq p^{*}$.
\end{remark}
We prove the following existence result in the spirit of Merle-Peletier \cite{Merle} to the case $p\neq2$.

\begin{theorem}\label{11224411}
Let $N\ge 2$, $1<p<N$ and $p^*<q<l$.
Equation
\eqref{32} admits a groundstate solution $ u_{0} \in D^{1,p}(\mathbb{R}^{N})\cap L^{l}(\mathbb{R}^{N})\cap C^{1,\alpha}_{loc}(\mathbb{R}^{N}),$
such that $u_0(x)$ is a positive monotone decreasing function of $|x|$ and
\begin{equation}\label{e-decay-supercrit}
u_{0}(x)\sim |x|^{-\frac{N-p}{p-1}} \quad \text{as}\quad |x| \rightarrow \infty.
\end{equation}

\end{theorem}
\noindent

\begin{remark} \label{Tangremark} The uniqueness result of \cite{Tang Moxun} is applicable to fast decay solutions to (\ref{32}). However the regularity hypothesis H1 as stated at p. 155 in \cite{Tang Moxun} would require $p^*\geq 2$, namely $p\geq \frac{2N}{N+2}.$ 
\end{remark}

\begin{proof}
Following Berestycki-Lions \cite{Berestycki-Lions} in the present zero-mass case context we solve the variational problem in $D^{1,p} (\mathbb{R}^{N})$ namely

\begin{equation}\tag{$S_{0}$}\label{991}
 S_{0}:= \inf \Big\{\int\limits_{\mathbb{R}^{N}}|\nabla w|^{p}dx\Big| w\in D^{1,p}(\mathbb{R}^{N}), \quad p^{*}\int\limits_{\mathbb{R}^{N}} \tilde F_{0}(w)dx=1 \Big\},
 \end{equation}
where
$$ \tilde F_{0}(w)= \int_{0}^{w} \tilde f_{0}(s) ds,$$
and $\tilde f_0(s)$ is a bounded truncation of the nonlinearity
$$f_0(s)=|s|^{q-2} s - |s|^{l-2} s,$$
e. g. 

\begin{equation}
\tilde{f_{0}}(u)=\begin{cases}
0, &u<0, \\
u^{q-1}-u^{l-1}, & u \in [0,1], \\
0,& u>1. \\
\end{cases}
\end{equation}
The above bounded truncation makes the minimisation problem well-posed in $D^{1,p}(\mathbb{R}^{N})$. Arguing as for the positive mass case the existence of a radially decreasing optimiser $u$ is standard. \newline
The global boundedness of the truncation allows to use the classical result of DiBenedetto, see e.g. Corollary p. 830 in \cite{DiBenedetto}, to show that $u\in C_{loc}^{1,\alpha}(\mathbb{R}^{N})$. \newline
Then  the maximum principle implies that any solution for the truncated problem solves in fact \eqref{32} and is strictly positive. \newline
Note that by Ni's inequality \ref{4.5.} and the $C_{loc}^{1,\alpha}(\mathbb{R}^{N})$ regularity it follows that $u\in L^{\infty}(\R^N).$ By interpolation with Sobolev's inequality this implies that $u\in L^{l}(\R^N)$ for all $l>p^*.$ \newline With the lemmas below on the asymptotic decay we conclude the proof.
\end{proof}
The following lemma about asymptotic properties of solutions is taken from \cite{Pinchover}.

\begin{lemma}[\cite{Pinchover}*{Corollary $8.3.$}]
Let $1<p<N $.
Assume that
$$ |V(x)|\leq \frac{g(|x|)}{1+|x|^{p}},$$
where $g:\mathbb{R^{+}\rightarrow \mathbb{R^{+}}} $ is bounded and continuous and satisfies the following conditions:
\begin{enumerate}
\item[$(C{1})$] $\Big|\int\limits_{1}^{\infty} \Big |t^{1-N} \int\limits_{1}^{t} \frac{g(|x|)}{|x|^{p}} |x|^{N-1} d|x| \Big|^{\frac{1}{p-1}}dt\Big| <\infty. $

\item[$(C{2})$] $\Big|\int\limits_{1}^{\infty} \frac{g(|x|)}{|x|}d|x|\Big|<\infty.$
\end{enumerate}
Assume that
\begin{equation}\label{02040}
- \Delta_{p}u +V(x)u^{p-1}=0, \qquad \text{in} \; \mathbb{R}^{N}\backslash B_{1}(0),
\end{equation}
admits a positive supersolution. Then \eqref{02040} admits a solution which satisfies
\begin{equation}\label{002040}
U_{0}(x)\sim  |x|^{-\frac{N-p}{p-1}} \qquad \text{as} \; |x|\rightarrow\infty.
\end{equation}
\end{lemma}

\begin{corollary}
If $$V(x)=\frac{c}{(1+|x|)^{p+\delta}},$$
and $c$ is sufficiently small then \eqref{02040} admits a positive solution that satisfies \eqref{002040}
\end{corollary}

\begin{proof}
We can take
$$g(|x|)=|x|^{-\delta}.$$
Then $(C1)$, $(C2)$ are elementary to check.
\end{proof}
The decay estimate \eqref{e-decay-supercrit} is proved in the following lemma.
\begin{lemma}\label{decaylemma}
Let $ u_{0} \in D^{1,p}(\mathbb{R}^{N})\cap L^{l}(\mathbb{R}^{N})$ be a positive radial solution of \eqref{32}.  Then
\begin{equation}\label{02042}
u_{0}(x)\sim|x|^{-\frac{N-p}{p-1}} \quad \text{as}\quad |x| \rightarrow \infty.
\end{equation}
\end{lemma}
\begin{proof}
Since  $ u_{0} \in D^{1,p}(\mathbb{R}^{N})\cap L^{l}(\mathbb{R}^{N})$ is radial then by the Ni type inequality \ref{4.5.}, we have
$$u_{0}\leq c|x|^{-\frac{N-p}{p}}, \quad \text{in} \; \mathbb{R}^{N}\backslash B_{1}(0). $$
ans since $l>p^{*}$ then we have for some $\delta_{1}> 0 $
\begin{equation} \label{36} u_{0}^{l-p}\leq c|x|^{-\frac{N-p}{p}(l-p)}=\frac{c}{|x|^{p+\delta _{1}}},\qquad \text{in} \; \mathbb{R}^{N}\backslash B_{1}(0),
\end{equation}
implying
$$u_{0}^{l-p} \leq\frac{C}{(1+|x|)^{p+\delta_{1}}}, \qquad \text{in} \; \mathbb{R}^{N},$$
for sufficiently large constant C independent of $x$.
Now set$$ -\Delta_{p}u_{0}+(u_{0}^{l-p})u^{p-1}_{0}=u_{0}^{q-1}\geq0, \quad \text{in} \; \mathbb{R}^{N},$$
and then we have$$ -\Delta_{p}u_{0}+\frac{C}{(1+|x|)^{p+\delta _{1}}}u^{p-1}_{0}\geq0, \qquad \text{in} \; \mathbb{R}^{N}.$$As a consequence, $u_{0} $ is a supersolution of \eqref{02040} and then by comparison principle (see Theorem \ref{A52} in the Appendix), we obtain
\begin{equation}\label{0241}
u_{0}\geq c|x|^{-\frac{N-p}{p-1}} \quad \text{in} \quad |x|>1.
\end{equation}
Similarly, we can set
$$ -\Delta_{p}u_{0}-(u_{0}^{q-p})u^{p-1}_{0}=-u_{0}^{l-1}\leq0,\qquad \text{in} \; \mathbb{R}^{N},$$
and since $q>p^{*} $ we have for some $\delta_{2}>0, $ $$u_{0}^{q-p}\leq c'|x|^{-\frac{N-p}{p}(q-p)}\leq\frac{c'}{|x|^{p+\delta_{2}}}, \qquad \text{in} \; |x|>1,$$
implying
$$u_{0}^{q-p} \leq\frac{C'}{(1+|x|)^{p+\delta_{2}}}, \qquad \text{in} \; \mathbb{R}^{N},$$
and hence
$$ -\Delta_{p}u_{0}-\frac{C'}{(1+|x|)^{p+\delta_{2} }}u^{p-1}_{0}\leq0 \qquad \text{in} \; \mathbb{R}^{N}.$$

Now since $u_{0}\in D_{rad}^{1,p}(\mathbb{R}^{N}) $ is a subsolution of \eqref{02040}, then by Lemma \ref{A55} $ u_{0}$ satisfies condition $(S) $ and hence by comparison principle Theorem \ref{A52}, we have
\begin{equation}\label{0242}
u_{0}\leq c'|x|^{-\frac{N-p}{p-1}}   \qquad \text{in} \quad |x|>1,
\end{equation}
and hence from  \eqref{0241} and\eqref{0242} the conclusion follows.
\end{proof}

\section{Proof of Theorem \ref{s123} and Theorem \ref{main-cor}: critical case $q = p^{*}$}\label{6}
In this section we analyse the behaviour of the groundstates $u_\eps$ of equation \eqref{1s} as $\eps\to 0$ in the critical case $q=p^*$
and prove Theorem \ref{s123}. Although our approach follows the ideas of \cite{Moroz-Muratov}, the present $p$-Laplacian setting requires substantial modifications.

\subsection{Variational estimates for $ S_{\varepsilon} $}
Equivalently to the Sobolev constant \eqref{2p}, we consider the Rayleigh type Sobolev quotient

$$ \mathcal{S}_{*}(w):= \frac{\int\limits_{\mathbb{R}^{N}}|\nabla w|^{p}dx }{\big(\int\limits_{\mathbb{R}^{N}}| w|^{p^{*}}dx\big)^{(N-p)/N} }, \quad w\in D^{1,p}(\mathbb{R}^{N}), \quad w\neq0,$$
which is invariant with respect to the dilations $w_\lambda(x):=w(x/\lambda)$,
so that $$ S_{*}= \displaystyle\inf_{0\neq w\in D^{1,p}(\mathbb{R}^{N})} \mathcal{S}_{*}(w). $$

We define the gap
\begin{equation}\label{e2.33i}
\sigma_{\varepsilon}:= S_{\varepsilon}-S_{*}.
\end{equation}
To estimate $ \sigma_{\varepsilon} $ in terms of $ \varepsilon $, we shall use the Sobolev minimizers $ W_{\mu} $ from \eqref{234} as test functions for \eqref{e2.4}. Since $W_{\lambda}\in L^{p}(\mathbb{R}^{N}) $ only if $1<p<\sqrt{N}$, we analyse the higher and lower dimensions separately. It is easy to check that $ W_{\lambda} \in L^{s}(\mathbb{R}^{N}) $ for all $s> \frac{N(p-1)}{N-p}$, with
$$||W_{\lambda}||_{s}^{s}=\lambda^{-\frac{N-p}{p}s+N} ||W_{1}||_{s}^{s}=\lambda^{-\frac{N-p}{p}(s-p^{*})}||W_{1}||^{s}_{s},$$
and that, if $1<p<\sqrt{N}$ then $W_{\lambda} \in L^{p}(\mathbb{R}^{N} )$ it holds that
$$||W_{\lambda}||_{p}^{p}=\lambda^{p} ||W_{1}||_{p}^{p}.$$
In the case of dimensions $p=\sqrt{N} $ and $ \sqrt{N} <p<N $, given $ R\gg\mu $, we introduce a cut-off function $ \eta_{R} \in C^{\infty}_{0}(\mathbb{R})$ such that $ \eta_{R}(r)=1 $ for $ |r|<R $, $ 0<\eta_{R}<1 $ for $ R<|r|<2R $, $ \eta_{R}(r)=0 $ for $ |r|>2R $ and $ |\eta^{'}_{R}(r)|\leq 2/R $. We then compute as in e.g.\cite{Struwe}*{Chapter III, proof of Theorem $2.1$}
\begin{eqnarray}\label{eq3.38}
||\nabla\big(\eta_{R} W_{\mu}(x)\big)||^{p}_{p}&=&S_{*}+O\Big(\Big(\frac{R}{\mu}\Big)^{-\frac{N-p}{p-1}} \Big),\\
\label{eq3.399}
||\eta_{R}W_{\mu}||^{p^{*}}_{p^{*}}&=&1-O\Big(\Big({\frac{R}{\mu}}\Big)^{-\frac{N}{p-1}}\Big),\\
\label{eq3.39}
||\eta_{R}W_{\mu}||^{l}_{l}&=&\mu^{-\frac{N-p}{p}(l-p^{*})}|| W_{1}||^{l}_{l}\Big(1- O\Big(\Big(\frac{R}{\mu}\Big)^{-\frac{(N-p)l}{p-1}+N}\Big),
\end{eqnarray}
and
\begin{equation}\label{eq3.398}
||\eta_{R}W_{\mu}||^{p}_{p}=\begin{cases}
O\Big(\mu^{p} \log R\Big), & p=\sqrt{N},\\
O \Big(\mu^{\frac{N-p}{p-1}}R^{\frac{p^{2}-N}{p-1}}\Big), & \sqrt{N}<p<N.\\
\end{cases}
\end{equation}
As a consequence of these expansions we get an upper estimate for  $\sigma_{\varepsilon}$ which plays a key role in what follows.
\begin{lemma}\label{3.6}
	It holds that
	\begin{equation}\label{eq3.6}
		0<\sigma_{\varepsilon}\lesssim\begin{cases}
		\varepsilon^{\frac{l-p^{*}}{l-p}} &  \;1<p<\sqrt{N},\\
		\varepsilon^{\frac{(N-p)(l-p^{*})} {p[(l-p^{*})(p-1)+p]}} & \sqrt{N}<p<N,\\
		\Big(\varepsilon (\log \frac{1}{\varepsilon})\Big)^{\frac{(l-p^{*})}{(l-p)}} & p=\sqrt{N}.\\
		\end{cases}
	\end{equation}
	Hence, $ \sigma_{\varepsilon}\rightarrow 0 $ as $ \varepsilon\rightarrow 0 $.
\end{lemma}

\begin{proof}
We first observe that since $$S_{*}\leq \mathcal{S}_{*}(w_{\varepsilon})<\mathcal{S}_{\varepsilon}(w_{\varepsilon})=S_{\varepsilon},  $$
it follows that $ \sigma_{\varepsilon} > 0.$ We now obtain the upper bounds on $ \sigma_{\varepsilon} $.
		
	\noindent
	{\em Case $1<p<\sqrt{N}$.} Note that $ W_{\mu}\in\mathcal{M_{\varepsilon}} $ for all sufficiently small $ \varepsilon $ and sufficiently large $ \mu $, and we have
	\begin{equation}\label{p33}
		\mathcal{S}_{\varepsilon}(W_{\mu})\leq \frac{S_{*} }{\Big(1 -\varepsilon\mu^{p}\beta_{p}-\mu^{\frac{-(N-p)}{p}(l-p^{*})}\beta_{l}\Big)^{(N-p)/N} },
	\end{equation}
	where $$\beta_{p}:=\frac{p^{*} }{p}||W_{1}||_{p}^{p}, \qquad \beta_{l}:=\frac{p^{*}}{l}||W_{1}||_{l}^{l}.$$
	We now optimise the right hand side of the estimate \eqref{p33} picking $\mu$ such that the function
	\begin{equation*}
	\psi_{\varepsilon}(\mu)  :=\beta_{p}\varepsilon\mu^{p}+\beta_{l}\mu^{-\frac{N-p}{p}(l-p^{*})}.
	\end{equation*}	
achieves its minimum. This occurs at
	\begin{equation}\label{p44}
	\mu_{\varepsilon}\sim\varepsilon^{-\frac{p}{(N-p)(l-p)}}
	\end{equation}
	and we have
	$$ \displaystyle\min_{\mu>0}\psi_\eps\sim\psi_{\varepsilon}(\mu_{\varepsilon})\sim\varepsilon^{\frac{l-p^{*}}{l-p}}. $$	
	In the present case $1<p<\sqrt{N}$, we may conclude that
	\begin{equation}
	\mathcal{S}_{\varepsilon}(W_{\mu})\lesssim \frac{S_{*} }{\Big(1 -\psi_{\varepsilon}(\mu_{\varepsilon})\Big)^{(N-p)/N} } = S_{*}\Big(1+O(\psi_{\varepsilon}(\mu_{\varepsilon}))\Big)
= S_{*}+O\Big(\varepsilon^{\frac{l-p^{*}}{l-p}}\Big),
	\end{equation}
	and \eqref{p44} is the value of $ \mu_{\varepsilon}$ such that the bound \eqref{eq3.6} is achieved on the function $ W_{\mu_{\varepsilon}} .$  
	
	\noindent
	{\em Case $ p>\sqrt{N} $.}
	We assume here that $ R\gg\mu $. Using $ \eta_{R}W_{\mu} $ as test function and using the calculation in \eqref{eq3.38}-\eqref{eq3.398}, we get
	\begin{multline*}
	\mathcal{S}_{\varepsilon}(\eta_{R}W_{\mu})\leq
	\Big(S_{*}+O\Big(\Big(\frac{R}{\mu}\Big)^{-\frac{N-p}{p-1}} \Big)\Big)\\
	\times\Big( 1-\Big\{O\Big(\Big({\frac{R}{\mu}}\Big)^{-\frac{N}{p-1}}\Big)+\varepsilon O \Big(\mu^{\frac{N-p}{p-1}}R^{\frac{p^{2}-N}{p-1}}\Big)\\+ \mu^{-\frac{N-p}{p}(l-p^{*})}||W_{1}||^{l}_{l}\Big[1- O\Big(\Big(\frac{R}{\mu}\Big)^{-\frac{(N-p)l}{p-1}+N}\Big)\Big]\Big\}\Big)^{-\frac{N-p}{N}},
	\end{multline*}
	and hence as $\frac{R}{\mu}\rightarrow \infty$, we have
	$$\mathcal{S}_{\varepsilon}(\eta_{R}W_{\mu})\leq S_{*}\Big(1+\psi_{\varepsilon}(\mu,R)\Big),$$
	where
	
		\begin{equation}\label{epbsq}
	\psi_{\varepsilon}(\mu,R):=\Big({\frac{R}{\mu}}\Big)^{-\frac{N-p}{p-1}}+\varepsilon \mu^{\frac{N-p}{p-1}} R^{\frac{p^{2}-N}{p-1}}+ \mu^{-\frac{N-p}{p}(l-p^{*})}.
	\end{equation}	
	If in particular we choose
	\begin{equation}\label{eq3.13}
\mu_{\varepsilon}=\varepsilon^{-\frac{1} {(l-p^{*})(p-1)+p}},\qquad	R_{\varepsilon}=\varepsilon^{-\frac{1}{p} }.
	\end{equation}	
	we then find that
	$$\psi_{\varepsilon}(\mu_{\varepsilon},R_{\varepsilon})\sim \varepsilon^{\frac{(N-p)(l-p^{*})} {p[(l-p^{*})(p-1)+p]},}$$
	and, similarly to the above case, the bound \eqref{eq3.6} is achieved on the test function $ \eta_{R_{\varepsilon}}W_{\mu_{\varepsilon}} $ provided $ \mu_{\varepsilon} $ and $ R_{\varepsilon} $ are as in \eqref{eq3.13}.
	
	\noindent
	{\em Case $ p=\sqrt{N}$.} Again we assume that $ R\gg\mu $. Testing again against $ \eta_{R}W_{\mu} $ and by\eqref{eq3.38}-\eqref{eq3.398} with $ p=\sqrt{N}$,
	we get
	\begin{multline*}
	\mathcal{S}_{\varepsilon}(\eta_{R}W_{\mu})\leq \Big(S_{*}+O\Big(\Big(\frac{R}{\mu}\Big)^{-\frac{N-p}{p-1}}\Big)\Big)\\
	\times\Big( 1-\Big(O\Big(\Big({\frac{R}{\mu}}\Big)^{\frac{-N}{p-1}}\Big)+\varepsilon  O\Big(\mu^{p}\log R\Big)\\+ \mu^{-\frac{N-p}{p}(l-p^{*})}||W_{1}||^{l}_{l}\Big[1- O\Big(\Big(\frac{R}{\mu}\Big)^{-\frac{(N-p)l}{p-1}+N}\Big)\Big]\Big)\Big)^{-(N-p)/N},
	\end{multline*}
	and then as $\frac{R}{\mu}\rightarrow \infty$, we have
	$$
	\mathcal{S}_{\varepsilon}(\eta_{R}W_{\mu})\leq S_{*}\Big(1+\psi_{\varepsilon}(\mu,R)\Big),$$
	where
	\begin{equation}\label{33233}
	\psi_{\varepsilon}(\mu,R):=\Big({\frac{R}{\mu}}\Big)^{-\frac{N-p}{p-1}}+\varepsilon \mu^{p}\log R+ \mu^{-\frac{N-p}{p}(l-p^{*})}.
	\end{equation}	
	Choose
	\begin{equation}\label{3.110}
	R_{\varepsilon}:=\varepsilon^{-\frac{1}{p}},
	\qquad \mu_{\varepsilon}:= \Big(\varepsilon \log \frac{1}{\varepsilon}\Big)^{\frac{-p}{(N-p)(l-p)}},
	\end{equation}
and hence
	$$  \psi_\eps(\mu_{\varepsilon},R_{\varepsilon})\sim \Big(\varepsilon \log \frac{1}{\varepsilon}\Big)^{\frac{(l-p^{*})}{(l-p)}}.$$	
Thus the bound \eqref{eq3.6} is achieved by the test function $ \eta_{R_{\varepsilon}}W_{\mu_{\varepsilon}} $, where $ \mu_{\varepsilon} $ and $ R_{\varepsilon} $ are defined in \eqref{3.110}.
\end{proof}

\subsection{Poho\v{z}aev estimates }
For $\eps\in(0,\eps_*)$, let $w_{\varepsilon}>0$ be a family of the minimizers for \eqref{e2.4} (or equivalently \eqref{p29}).
This minimizers $w_{\varepsilon}$ solve the Euler Lagrange equation
\begin{equation}\label{e3.7-2}
-\Delta_{p} w_{\varepsilon} = S_{\varepsilon}\big(-\eps w_{\varepsilon}^{p-1}+w_{\varepsilon}^{p^*-1}-w_{\varepsilon}^{l-1}\big) \quad  \text{in} \;  \mathbb{R}^{N}
\end{equation}
with the original (untruncated) nonlinearity.

Our next step is to use Nehari's identity combined with Poho\v{z}aev's identity for \eqref{e3.7-2} in order to obtain the following useful relations between the norms of $w_\eps$.


\begin{lemma}\label{314}
	For all $1<p<N$, set $ k:=\frac{l(p^{*}-p)}{p(l-p^{*})}>0 $. Then, it holds that
	$$ ||w_{\varepsilon}||^{l}_{l}= k\varepsilon ||w_{\varepsilon} ||_{p}^{p},$$
	$$ ||w_{\varepsilon}||^{p^{*}}_{p^{*}}=1+ (k+1)\varepsilon ||w_{\varepsilon} ||_{p}^{p}.$$
\end{lemma}
\begin{proof}
	Since $w_{\varepsilon}$ is a minimizer of \eqref{e2.4}, identities \eqref{p24}-\eqref{p25} read	
	\begin{equation}\label{e2.34c}
	1=  ||w_{\varepsilon}||_{p^{*}}^{p^{*}}  -\varepsilon ||w_{\varepsilon}||_{p}^{p}- ||w_{\varepsilon}||_{l}^{l} , \quad 	1 =  ||w_{\varepsilon}||_{p^{*}}^{p^{*}}- \frac{p^{*}\varepsilon}{p}  ||w_{\varepsilon}||_{p}^{p}-\frac{p^{*}}{l} ||w_{\varepsilon}||_{l}^{l} .
	\end{equation}
An easy calculation yields the conclusion.	
\end{proof}

\begin{lemma}\label{lem2.3.5}
	For all $1<p<N$, we have
	$$ \varepsilon (k+1)||w_{\varepsilon}||^{p}_{p}\leq \frac{N}{N-p} S^{-1}_{*} \sigma_{\varepsilon}\big(1+ o(1)\big).$$
\end{lemma}
\begin{proof}
	Using that $ w_{\varepsilon} $ is a minimizer for \eqref{e2.4}, by Lemma \ref{314} it follows that $$S_{*} \leq \mathcal{S}_{*} (w_{\varepsilon}) = \frac{|| \nabla w_{\varepsilon}||^{p}_{p}}{||w_{\varepsilon}||^{p}_{p^{*}}} =\frac{S_{\varepsilon}}{\Big(1+(k+1)\varepsilon ||w_{\varepsilon}||_{p}^{p}\Big)^{(N-p)/N} },$$
	namely,
	$$ S^{N/(N-p)}_{*}\Big(1+(k+1)\varepsilon ||w_{\varepsilon}||_{p}^{p}\Big)\leq S_{\varepsilon}^{N/(N-p)}.$$
	Setting $\sigma_{\varepsilon}:=S_{\varepsilon}-S_{*}$, as $ \varepsilon\rightarrow 0 $ we obtain	
	$$ S^{N/(N-p)}_{*}(k+1)\varepsilon ||w_{\varepsilon}||_{p}^{p}\leq  \sigma_{\varepsilon}\frac{N}{N-p}S^{\frac{N}{N-p}-1}_{*} +o(\sigma_{\varepsilon}),$$
	and this concludes the proof.
\end{proof}
	\noindent
We note that the above results allow us to understand the behavior of the norms associated with the minimizer $ w_{\varepsilon} $ to \eqref{e2.4}. In fact we have the following

\begin{corollary}\label{316}
	As $ \varepsilon\rightarrow 0 $, we have
	$$ \varepsilon||w_{\varepsilon}||^{p}_{p}\rightarrow 0,  \quad ||w_{\varepsilon}||_{l}^{l}\rightarrow 0,  \quad  ||w_{\varepsilon}||^{p^{*}}_{p^{*}}\rightarrow 1. $$
\end{corollary}
\subsection{Optimal rescaling}
We are now in a position to introduce an optimal rescaling which captures the convergence of the minimizers $w_\eps$ to the limit Emden-Fowler optimiser $W_1$.

\noindent Following \cite{Struwe}*{pp.38 and 44}, consider the concentration function
$$ Q_{\varepsilon}(\lambda)=\int\limits_{B_{\lambda}}|w_{\varepsilon}|^{p^{*}}dx,$$
where $ B_{\lambda}$ is the ball of radius $ \lambda$ centred at the origin. Note that $Q_{\varepsilon}(\cdot)$ is strictly increasing, with
$$\lim\limits_{\lambda\rightarrow 0} Q_{\varepsilon}(\lambda)=0, $$ and
$$\lim\limits_{\lambda\rightarrow \infty} Q_{\varepsilon}(\lambda)=||w_{\varepsilon}||^{p^{*}}_{p^{*}}\rightarrow1, \qquad \text{as}\; \varepsilon \rightarrow 0,$$ by Corollary \ref{316}. It follows that the equation $Q_{\varepsilon}(\lambda)=Q_{*}$ with
$$Q_{*}:=\int\limits_{B_{1}}|W_{1}(x)|^{p^{*}} dx<1,  $$
has a unique solution $\lambda=\lambda_{\varepsilon}> 0 $ for $\varepsilon\ll1,$ namely
\begin{equation}\label{319}
Q_{\varepsilon}(\lambda_{\varepsilon})=Q_{*}.
\end{equation}
By means of the value of $\lambda_{\varepsilon} $ implicitly defined by \eqref{319}, we set
\begin{equation}\label{p55}
v_{\varepsilon}(x):= \lambda^{\frac{N-p}{p}}_{\varepsilon}w_{\varepsilon}(\lambda_{\varepsilon}x),
\end{equation}
and easily check that
\begin{equation}\label{*e2.41}
||v_{\varepsilon}||_{p^{*}}=||w_{\varepsilon}||_{p^{*}}=1+o(1), \quad ||\nabla v_{\varepsilon}||^{p}_{p}=||\nabla w_{\varepsilon}||^{p}_{p}=S_{*}+o(1).
\end{equation}
namely $(v_{\varepsilon}) $ is a minimizing family for \eqref{2p}. Moreover
$$\int\limits_{B_{1}}|v_{\varepsilon}(x)|^{p^{*}}dx=Q_{*}. $$
The following convergence lemma follows by the Concentration-Compactness Principle of P.-L. Lions \cite{Struwe}*{Theorem $4.9 $}.
\begin{lemma}\label{2.3.6}
	For all $1<p<N$, it holds that
	$$||\nabla(v_{\varepsilon}-W_{1})||_{p}\rightarrow0,$$and $$||v_{\varepsilon}-W_{1}||_{p^{*}}\rightarrow0,$$as $\varepsilon\rightarrow 0.$
\end{lemma}
\begin{proof}
	By \eqref{*e2.41}, for any sequence $\varepsilon_{n}\rightarrow 0 $ there exists a subsequence $(\varepsilon_{\acute{n}}) $ such that $(v_{\varepsilon_{\acute{n}}}) $ converges weakly in $D^{1,p}(\mathbb{R}^{N}) $ to some radial functions $w_{0} \in D^{1,p}(\mathbb{R}^{N}) $. By the Concentration-Compactness Principle \cite{Struwe}*{Theorem $4.9 $} applied to $||v_{\varepsilon} ||_{p^{*}}^{-1} v_{\varepsilon}$, we have in fact that $(v_{\varepsilon_{\acute{n}}}) $ converges to $w_{0} $ strongly in $D^{1,p}(\mathbb{R}^{N}) $ and $L^{p^*}(\mathbb{R}^{N}) $. Hence, $ ||w_{0}||_{p^{*}}=1 $ and therefore $w_{0} $ is a radial minimizer of \eqref{2p}, that is necessarily $w_{0}\in \{W_{\lambda}\}_{\lambda>0} $. Note that it also holds
	$$\int\limits_{B_{1}}|w_{0}(x)|^{p^{*}}dx=Q_{*}.$$
As a consequence $w_{0}=W_{1}$. Since the sequence $(\varepsilon_n)$ was arbitrary, the whole sequence $ (v_{n})$ converges to $W_{1}$ strongly in $D^{1,p} (\mathbb{R}^{N})$ and $L^{p^*}(\mathbb{R}^{N}), $ and this concludes the proof.
\end{proof}

\subsection{Rescaled equation estimates}
Our next step is to obtain upper and lower estimates on the rescaling function $\lambda_\eps$, which is implicitly determined by \eqref{319}.
\noindent The rescaled function $ v_{\varepsilon} $ introduced in \eqref{p55}
is such that
\begin{equation}\tag{$R^{*}_{\varepsilon}$}\label{3321}
-\Delta_{p}v_{\varepsilon} = S_{\varepsilon}\Big(-\varepsilon \lambda_{\varepsilon}^{p}v_{\varepsilon}^{p-1} +v_{\varepsilon}^{p^{*}-1}-\lambda_{\varepsilon}^{-(N-p)(\frac{l-p}{p})+p}v_{\varepsilon}^{l-1}\Big),
\end{equation}
as \eqref{e2.4} is achieved by $w_\varepsilon.$  By construction, for $v_{\varepsilon}$ we obtain
$$	||v_{\varepsilon}||_{l}^{l}=\lambda_{\varepsilon}^{\frac{p(l-p^{*})}{(p^{*}-p)}} ||w_{\varepsilon} ||_{l}^{l},
\qquad
||v_{\varepsilon}||_{p}^{p}=\lambda_{\varepsilon}^{-p} ||w_{\varepsilon} ||_{p}^{p}.
$$
Putting Lemmas \ref{314} and \ref{lem2.3.5} together we then achieve the relation
\begin{equation}\label{eq3.21}
\lambda_{\varepsilon}^{-\frac{p(l-p^{*})}{(p^{*}-p)}}||v_{\varepsilon} ||^{l}_{l} =\lambda_{\varepsilon}^{p}k\varepsilon||v_{\varepsilon}||_{p}^{p}\lesssim\sigma_{\varepsilon},
\end{equation}
which yields the following\\
\begin{lemma}\label{2.3.7} Let $1<p<N $. Then
	$$\sigma_{\varepsilon}^{-\frac{(p^{*}-p)}{p(l-p^{*})}} \lesssim \lambda_{\varepsilon} \lesssim \varepsilon^{-\frac{1}{p}} \sigma_{\varepsilon}^{\frac{1}{p}}.$$
\end{lemma}
\begin{proof}
The statement will follow by \eqref{eq3.21} combined with the observation that
	$$\displaystyle\liminf_{\varepsilon\rightarrow 0}||v_{\varepsilon}||_{l}>0, \quad \displaystyle\liminf_{\varepsilon\rightarrow 0}||v_{\varepsilon}||_{p}>0. $$
	The former is a consequence of Lemma \ref{2.3.6} and H\"older's inequality, which yields $ L^{l}(B_{1}) \subset L^{p^{*}}(B_{1}) $ since $l>p^{*},$  hence
	\begin{eqnarray*}
		c ||v_{\varepsilon} \mathcal{X}_{B_{1}}||_{l}\geq||v_{\varepsilon} \mathcal{X}_{B_{1}}||_{p^{*}}&\geq& ||W_{1}\mathcal{X}_{B_{1}}||_{p^{*}} -||(W_{1}-v_{\varepsilon}) \mathcal{X}_{B_{1}}||_{p^{*}}\\
		&=& ||W_{1}\mathcal{X}_{B_{1}}||_{p^{*}} -o(1).
	\end{eqnarray*}
	\noindent Here $ \mathcal{X}_{B_{R}} $ is the characteristic function of $ B_{R}$. To show the latter,  by the embedding $  L^{p^{*}}(B_{1}) \subset L^{p}(B_{1}) $ since $p^{*}>p  $, we obtain
	$$c ||v_{\varepsilon} \mathcal{X}_{B_{1}}||_{p^{*}}\geq||v_{\varepsilon} \mathcal{X}_{B_{1}} ||_{p}\geq
	||W_{1}\mathcal{X}_{B_{1}}||_{p} -||(W_{1}-v_{\varepsilon}) \mathcal{X}_{B_{1}}||_{p}=||W_{1}\mathcal{X}_{B_{1}}||_{p} -o(1),$$
	and this concludes the proof.
\end{proof}
\noindent
By \eqref{eq3.6} and Lemma \ref{2.3.7} we obtain both an estimate from below
\begin{equation}\label{3.22}
\lambda_{\varepsilon}\gtrsim\sigma_{\varepsilon}^{-\frac{(p^{*}-p)}{p(l-p^{*})}} \gtrsim \begin{cases}
	\varepsilon^{-\frac{(p^{*}-p)}{p(l-p)}} &  \;1<p<\sqrt{N},\\
	\varepsilon^{-\frac{1} {[(l-p^{*})(p-1)+p]}} & \sqrt{N}<p<N,\\
	\Big(\varepsilon (\log \frac{1}{\varepsilon})\Big)^{-\frac{(p^{*}-p)}{p(l-p)}} & p=\sqrt{N},\\
	\end{cases}
\end{equation}
and from above
\begin{equation}\label{3.23}
\lambda_{\varepsilon}\lesssim \begin{cases}
	\varepsilon^{-\frac{p^{*}-p}{p(l-p)}} &  \;1<p<\sqrt{N},\\
	\varepsilon^{-\frac{(p^{2}-N)(l-p^{*})+p^{2}} {p^{2}[(l-p^{*})(p-1)+p]}} & \sqrt{N}<p<N,\\
	\varepsilon^{-\frac{(p^{*}-p)}{p(l-p)}}\Big( \log \frac{1}{\varepsilon}\Big)^{\frac{(l-p^{*})}{p(l-p)}} & p=\sqrt{N}.
	\end{cases}
\end{equation}
We note that in the case $1<p<\sqrt{N}$ the above lower and upper estimates are equivalent, therefore we have the following

\begin{corollary}\label{partialbound}
Let $1<p<\sqrt{N}$. Then $||v_{\varepsilon}||_{l}$ and $||v_{\varepsilon}||_{p}$ are bounded.
\end{corollary}
\begin{proof}
This is an immediate consequence of \eqref{eq3.21}-\eqref{3.23}.
\end{proof}
\noindent
In the case $p\geq \sqrt N$ we take into account the growth of $||v_{\varepsilon}||_{p}$ to obtain matching bounds. In this case instead of \eqref{3.23} we use the more explicit upper bound
\begin{equation}\label{3.25}
\lambda_{\varepsilon}\lesssim \frac{\varepsilon^{-1/p}\sigma^{1/p}_{\varepsilon}}{||v_{\varepsilon}||_{p}}\lesssim ||v_{\varepsilon}||^{-1}_{p} \begin{cases}
	\varepsilon^{-\frac{(p^{2}-N)(l-p^{*})+p^{2}} {p^{2}[(l-p^{*})(p-1)+p]}} & \sqrt{N}<p<N.\\
	\varepsilon^{-\frac{(p^{*}-p)}{p(l-p)}}\Big( \log \frac{1}{\varepsilon}\Big)^{\frac{(l-p^{*})}{p(l-p)}}& p=\sqrt{N},\\
	\end{cases}
\end{equation}
which follows from \eqref{eq3.21}  and \eqref{eq3.6}.

\subsection{A lower barrier for $p\ge 2$}

To refine the upper bound \eqref{3.23} we shall construct a lower barrier for $w_\eps$ in the critical regim\'es $\sqrt{N}\le p<N$.
For $p\ge 2$ this will be done using the following uniform estimate.

\begin{lemma}\label{p93}
Given $\mu>0$ and $\gamma>0$, set
$$h(r):=r^{-\gamma}e^{-\mu r}.$$
Assume that $p\ge 2$ and that $N-1-2\gamma(p-1)\le 0$ and $\gamma(N-p-\gamma(p-1))\le 0$.
Then for all $\mu>0$ and $r>0$,
\begin{multline}\label{eeebbb}
-\Delta_p h+\mu^p(p-1)h^{p-1}\\
\leq\mu\frac{\gamma^{p-2}(N-1-2\gamma(p-1))}{r^{p-1}}h^{p-1}+\frac{\gamma^{p-1}(N-p-\gamma(p-1))}{r^p}h^{p-1}.
\end{multline}
\end{lemma}

\begin{remark}
If $p=2$ then \eqref{eeebbb} becomes an equality.
\end{remark}

\begin{proof}
By direct calculations, we have
\begin{multline*}
-\Delta_p h+\mu^p(p-1)h^{p-1}=  (p-1)\mu^2\left\{\mu^{p-2}-\Big(\mu+\frac{\gamma}{r}\Big)^{p-2}\right\}h^{p-1}\\
+ \Big(\mu+\frac{\gamma}{r}\Big)^{p-2}
\left\{\mu\frac{N-1-2\gamma(p-1)}{r}+\frac{\gamma(N-p-\gamma(p-1))}{r^2}\right\}h^{p-1}.
\end{multline*}

\noindent
For all $\mu>0$ and $r>0$, by monotonicity we have
$$\left\{\mu^{p-2}-\Big(\mu+\frac{\gamma}{r}\Big)^{p-2}\right\}\le 0,$$
$$\Big(\mu+\frac{\gamma}{r}\Big)^{p-2}\ge \Big(\frac{\gamma}{r}\Big)^{p-2}.$$
Therefore, assuming that $N-1-2\gamma(p-1)\le 0$ and $\gamma(N-p-\gamma(p-1))\le 0$ we can estimate,
\begin{multline*}
-\Delta_p h+\mu^p(p-1)h^{p-1}\le \Big(\frac{\gamma}{r}\Big)^{p-2}\left\{\mu\frac{N-1-2\gamma(p-1)}{r}+\frac{\gamma(N-p-\gamma(p-1))}{r^2}\right\}h^{p-1}\\
\le \mu\frac{\gamma^{p-2}(N-1-2\gamma(p-1))}{r^{p-1}}h^{p-1}+\frac{\gamma^{p-1}(N-p-\gamma(p-1))}{r^p}h^{p-1},
\end{multline*}
uniformly for all $\mu>0$ and $r>0$.
\end{proof}

\begin{remark}
In the case $1<p<2$ by monotonicity, convexity and Taylor for all $\mu>0$ and $r>0$ we have
$$0\le\left\{\mu^{p-2}-\Big(\mu+\frac{\gamma}{r}\Big)^{p-2}\right\}\le (2-p)\mu^{p-3}\frac{\gamma}{r}.$$
Similarly, we can estimate
\begin{equation}\label{est-mu}
\Big(\mu+\frac{\gamma}{r}\Big)^{p-2}\ge \mu^{p-2}-(2-p)\mu^{p-3}\frac{\gamma}{r},
\end{equation}
or, alternatively,
\begin{equation}\label{est-r}
\Big(\mu+\frac{\gamma}{r}\Big)^{p-2}\ge \Big(\frac{\gamma}{r}\Big)^{p-2}-(2-p)\Big(\frac{\gamma}{r}\Big)^{p-3}\mu.
\end{equation}
Therefore, assuming that $N-1-2\gamma(p-1)\le 0$ and $\gamma(N-p-\gamma(p-1))\le 0$ we can estimate,
\begin{multline}
-\Delta_p h+\mu^p(p-1)h^{p-1}\le  \mu^{p-1}\frac{(2-p)(p-1)\gamma}{r}h^{p-1}\\
+ \Big(\quad\text{\eqref{est-mu} or \eqref{est-r}}\quad\Big)
\left\{\mu\frac{N-1-2\gamma(p-1)}{r}+\frac{\gamma(N-p-\gamma(p-1))}{r^2}\right\}h^{p-1}.\\
\label{wrong-est}
\end{multline}
Both \eqref{est-mu} and \eqref{est-r} introduce a large positive term in \eqref{wrong-est} which we cannot control.
\end{remark}
\noindent To estimate the norm $ ||v_{\varepsilon}||_{p}$, we note that
$$ -\Delta_{p}v_{\varepsilon}+S_{\varepsilon}\varepsilon \lambda_{\varepsilon}^{p}|v_{\varepsilon}|^{p-1}= S_{\varepsilon}  |v_{\varepsilon}|^{p^{*}-1} -  S_{\varepsilon}  \lambda_{\varepsilon}^{-(N-p)\frac{l-p}{p}+p}|v_{\varepsilon}|^{l-1}\geq -V_{\varepsilon}(x)v^{p-1}_{\varepsilon},$$
where we have set
$$ V_{\varepsilon}(x):= S_{\varepsilon} \lambda_{\varepsilon}^{-(N-p)\frac{l-p}{p}+p} v_{\varepsilon}^{l-p}(x).$$
By the radial decay estimate \eqref{p129}
we have $$ v_{\varepsilon}(x)\leq C_{N,p^{*}} |x|^{-N/p^{*}}||v_{\varepsilon}||_{p^{*}} .$$
By \eqref{*e2.41} and since $ \lambda_{\varepsilon}^{{-\frac{p(l-p^{*})}{(p^{*}-p)}} } \lesssim \sigma_{\varepsilon}\rightarrow 0  $  Lemmas \ref{3.6} and \ref{2.3.7} yield, for sufficiently small $ \varepsilon>0,$ the following decay estimate
\begin{eqnarray*}
V_{\varepsilon}(x):= S_{\varepsilon} \lambda_{\varepsilon}^{-(N-p)\frac{l-p}{p}+p} v_{\varepsilon}^{l-p}(x)
&\leq& S_{\varepsilon} \lambda_{\varepsilon}^{-(N-p)\frac{l-p}{p}+p} c^{l-p}_{p^{*}}||v_{\varepsilon}||^{l-p}_{p^{*}}|x|^{-\frac{N}{p^{*}}(l-p)}\\
&\leq& C\lambda_{\varepsilon}^{-(N-p)\frac{l-p}{p}+p} |x|^{-(p+\delta)} ,
\end{eqnarray*}
where $\delta:=\frac{N-p}{p}(l-p)-p>0$ and the constant $ C>0 $ does not depend on $ \varepsilon $ or $ x $.
Hence, for small $ \varepsilon>0 $ the rescaled functions $ v_{\varepsilon}>0 $ satisfy the homogeneous inequality
\begin{equation}\label{3.15*}
-\Delta_{p}v_{\varepsilon}+S_\varepsilon \varepsilon \lambda_{\varepsilon}^{p} v_{\varepsilon}^{p-1}+V_{\eps}(x)v^{p-1}_{\varepsilon}\geq 0, \quad x\in \mathbb{R^{N}}.
\end{equation}
The following result provides a suitable lower barrier to \eqref{3.15*} below.

\begin{lemma}\label{3.1.9}
Assume $N\ge 4$ and $2\le p<\frac{N+1}{2}$. Then there exists $R>0$, independent on $\eps>0$,
such that for all small $\eps> 0$,
$$h_{\varepsilon}(x):=|x|^{-\frac{N-p}{p-1}} e^{-\sqrt[p]{\varepsilon S_{\varepsilon}}\lambda_{\varepsilon}|x|} $$
satisfies
\begin{equation}\label{3.15*}
-\Delta_{p}h_{\varepsilon}+(p-1)S_\varepsilon \varepsilon \lambda_{\varepsilon}^{p} h_{\varepsilon}^{p-1}+V_{\varepsilon}(x)h_\eps^{p-1}\leq 0, \quad |x|>R.
\end{equation}
\end{lemma}
\begin{proof}
By Lemma \ref{p93} with $\gamma=\frac{N-p}{p-1}$ we conclude that there exists $R>1$, independent of $\eps>0$, such that
\begin{multline*}
-\Delta_p h_{\varepsilon}+(p-1)S_\varepsilon \eps\lambda_\eps^p h_{\varepsilon}^{p-1}+
V_\eps(x)h_{\varepsilon}^{p-1}\\
\leq
(\varepsilon S_\varepsilon)^{\frac{1}{p}}\lambda_{\varepsilon}\frac{\gamma^{p-2}(N-1-2\gamma(p-1))}{r^{p-1}}h_{\varepsilon}^{p-1}+\lambda_\eps^{-\frac{N-p}{p}(l-p)+p}\frac{C}{r^{p+\delta}}h_{\varepsilon}^{p-1}\\
\leq \left\{-\gamma^{p-2}\big(N+1-2p\big)(\varepsilon S_\varepsilon)^{\frac{1}{p}}\lambda_{\varepsilon}+C\lambda_\eps^{-\frac{N-p}{p}(l-p)+p}\right\}
\frac{1}{r^{p-1}}h_{\varepsilon}^{p-1}\quad\text{for $|x|>R$}.
\end{multline*}
\noindent It is convenient to denote $s:=l-p^*>0$.
Taking into account that $-\frac{N-p}{p}(l-p)+p=s(1-N/p)<0$,
we can use the lower bound \eqref{3.22} on $\lambda_\varepsilon$ to estimate
\begin{multline*}
	-\gamma^{p-2}\big(N+1-2p\big)(\varepsilon S_\varepsilon)^{\frac{1}{p}}\lambda_{\varepsilon}+C\lambda_\eps^{-\frac{N-p}{p}(l-p)+p}\\
	\le-\gamma^{p-2}(N+1-2p)S_\varepsilon^{\frac{1}{p}}\varepsilon^{\frac{1}{p}-\frac{1} {[s(p-1)+p]}} +C\varepsilon^\frac{(N-p)(l-p)-p^2} {p[s(p-1)+p]}\\
	\le-\gamma^{p-2}(N+1-2p)S_\varepsilon^{\frac{1}{p}}\varepsilon^{\frac{s(p-1)} {p[s(p-1)+p]}} +C\varepsilon^\frac{(N-p)(l-p)-p^2} {p[s(p-1)+p]}
	\le 0,
	\end{multline*}
for all sufficiently small $\eps>0$, provided that $p<(N+1)/2$,
which completes the proof.
\end{proof}

\begin{lemma}\label{3.1.10}
Assume $N\ge 4$ and $2\le p<\frac{N+1}{2}$.
There exists $ R> 0 $ and $ c>0 $, independent on $\eps>0$, such that for all small $ \varepsilon >0$, 
$$ v_{\varepsilon}(x)\geq c|x|^{-\frac{N-p}{p-1}} e^{-\sqrt[p]{\varepsilon S_\varepsilon} \lambda_{\varepsilon} |x|} \quad (|x|>R).$$
\end{lemma}
\begin{proof}
Define the barrier
$$ h_{\varepsilon}(x):= |x|^{-\frac{N-p}{p-1}} e^{-\sqrt[p]{\varepsilon S_\varepsilon } \lambda_{\varepsilon}|x|},$$
which satisfies
\begin{equation}\label{3.15*}
-\Delta_{p}h_{\varepsilon}+\varepsilon S_{\varepsilon}\lambda_{\varepsilon}^{p} h_{\varepsilon}^{p-1}+V_{\varepsilon}(x)h_\eps^{p-1}\leq 0, \quad |x|>R.
\end{equation}
by Lemma \ref{3.1.9}.
\noindent Note that Lemma \ref{2.3.6} and Lemma \ref{1.14} in the Appendix imply
$$ ||(v_{\varepsilon}-W_{1})_{B_{R}\backslash B_{R/2}}||_{\infty}\rightarrow 0, $$
and hence
$$ v_{\varepsilon} (|x|) \rightarrow W_{1}(|x|), \qquad \text{for} \; |x|=R.$$
Hence for all sufficiently small $ \varepsilon >0, $
we have
$$ v_{\varepsilon}(R) \geq \frac{1}{2} W_{1}(R), \quad \text{for} \; |x|=R. $$
Since $ h_{\varepsilon}(R) $ is a monotone decreasing function in $ \varepsilon $, then by a suitable choice of a uniform small constant $ c>0 $ we obtain
$$   c h_{\varepsilon}(R)\leq \frac{1}{2}W_{1}(R),$$
and hence
$$v_{\varepsilon}(R) \geq c h_{\varepsilon}(R), \quad \text{for \; all \;small}\; \varepsilon> 0.$$
Then the homogeneity of \eqref{3.15*} implies
$$ -\Delta_{p}(ch_{\varepsilon})+\varepsilon S_\varepsilon \lambda_{\varepsilon}^{p} (ch_{\varepsilon})^{p-1}+V_{\varepsilon}(x)(ch_{\varepsilon})^{p-1}\leq 0, \quad \text{in} \; |x|>R,$$
for all small $ \varepsilon > 0 $.
Define a function $ ch_{\varepsilon,k} $ by
$$ ch_{\varepsilon,k}=ch_{\varepsilon}-k^{-1} < ch_{\varepsilon},\qquad \text{for\; all} \; k>0, $$
then
\begin{equation}\label{3.18}
-\Delta_{p}(ch_{\varepsilon,k})+V_{\varepsilon}(x)(ch_{\varepsilon,k})^{p-1}+\varepsilon S_\varepsilon \lambda_{\varepsilon}^{p} (ch_{\varepsilon,k})^{p-1}\leq 0,\quad \text{in} \; |x|>R
\end{equation}
and
$$v_{\varepsilon} \geq ch_{\varepsilon} > ch_{\varepsilon,k}, \qquad \text{for} \; |x|=R. $$
Now, since
$$ ch_{\varepsilon}\rightarrow 0, \qquad \text{as} \; |x|\rightarrow +\infty, $$
then for $ k $ large enough there exists $ R_{k} > R $ such that
$$  ch_{\varepsilon,k} =0, \qquad \text{for} \; |x|=R_{k}, $$
and since $ v_{\varepsilon} > 0$, then $$v_{\varepsilon} > ch_{\varepsilon,k}, \qquad \text{for}\; |x|=R_{k}.$$
As a consequence, from \eqref{3.15*} and \eqref{3.18}, using the comparison principle (see Theorem \ref{A52} in the Appendix) we obtain
$$v_{\varepsilon} \geq ch_{\varepsilon,k}, \qquad \text{for}\; R <|x|< R_{k},$$
which can be achieved for every $ k $. 
Since $ R_{k} \rightarrow \infty$ as $k\rightarrow \infty$, the assertion follows.
\end{proof}

\subsection{Critical dimensions $N\ge 4$ and $ \sqrt{N}\leq p<\frac{N+1}{2}$ completed}
We now apply Lemma \ref{3.1.10} to obtain matching estimates for the blow-up of $ ||v_{\varepsilon} ||_{p}$ in dimensions $N\ge 4$ and $ \sqrt{N}\leq p<\frac{N+1}{2}$ .
\begin{lemma}\label{3.13}
If $N\ge 4$ and $ \sqrt{N}< p<\frac{N+1}{2}$, then $||v_{\varepsilon}||_{p}^{p}\gtrsim \Big(\frac{1}{\sqrt[p]{\varepsilon } \lambda_{\varepsilon}}\Big)^{\frac{p^{2}-N}{p-1}}$.
\end{lemma}
\begin{proof}
Since $ \sqrt{N}< p<\frac{N+1}{2}$, we directly calculate from Lemma \ref{3.1.10}:
\begin{eqnarray*}
||v_{\varepsilon}||_{p}^{p} \geq \int\limits_{\mathbb{R}^{N}\backslash B_{R}} |v_{\varepsilon}|^{p} dx\geq
\int\limits_{R}^{\infty}r^{N-1}\big|c r^{-\frac{N-p}{p-1}} e^{-\sqrt[p]{\varepsilon S_\varepsilon} \lambda_{\varepsilon}r}\big|^{p} dr,
\end{eqnarray*}
and as $ \varepsilon \rightarrow 0$ (i.e. $ \frac{1}{\sqrt[p]{\varepsilon} \lambda_{\varepsilon}}\rightarrow \infty$), we have
$$||v_{\varepsilon}||_{p}^{p} \geq c^{p}\int\limits_{R}^{\frac{1}{\sqrt[p]{\varepsilon S_\varepsilon} \lambda_{\varepsilon}}} r^{\frac{p^{2}-N}{p-1}-1}  dr\geq \Big(\frac{C}{\sqrt[p]{\varepsilon} \lambda_{\varepsilon}}\Big)^{\frac{p^{2}-N}{p-1}},$$
and this completes the proof.
\end{proof}
\noindent
As an immediate consequence of the above result, by \eqref{3.25}, we obtain an upper estimate of $\lambda_{\varepsilon} $ which matches the lower bound of \eqref{3.22} in dimensions $N\ge 4$ and $ \sqrt{N}< p<\frac{N+1}{2}$ .
\begin{corollary}\label{31015}
If $N\ge 4$ and $ \sqrt{N}< p<\frac{N+1}{2}$, then $ \lambda_{\varepsilon} \lesssim \varepsilon^{-\frac{1} {[(l-p^{*})(p-1)+p]}}$.
\end{corollary}
\noindent
We now move to consider the case $p=\sqrt{N}$.
\begin{lemma}\label{3.15}
If $N\ge 4$ and  $ p=\sqrt{N}$ then it holds that $||v_{\varepsilon}||_{p}^{p}\gtrsim \log (\frac{1}{\sqrt[p]{\varepsilon} \lambda_{\varepsilon}})$.
\end{lemma}
\begin{proof}
Since $ p=\sqrt{N}$, by Lemma \ref{3.1.10} we immediately get

\begin{eqnarray*}
||v_{\varepsilon}||_{p}^{p}
&\geq& \int\limits_{\mathbb{R}^{N}\backslash B_{R}}r^{N-1}|v_{\varepsilon}(r)|^{p} dr
\geq
 c^{p}\int\limits_{R}^{\frac{1}{\sqrt[p]{\varepsilon S_\varepsilon} \lambda_{\varepsilon}}} r^{\frac{p^{2}-N}{p-1}-1}  dr\\
&=&c^{p}\int\limits_{R}^{\frac{1}{\sqrt[p]{\varepsilon S_\varepsilon} \lambda_{\varepsilon}}} r^{-1}  dr\geq \log(\frac{C}{\sqrt[p]{\varepsilon} \lambda_{\varepsilon}}),
\end{eqnarray*}
and this concludes the proof.
\end{proof}
\begin{corollary}\label{31017}
If $N\ge 4$ and  $ p=\sqrt{N}$ then it holds that $ \lambda_{\varepsilon} \lesssim \Big(\varepsilon (\log \frac{1}{\varepsilon})\Big)^{-\frac{(p^{*}-p)}{p(l-p)}}$.
\end{corollary}
\begin{proof}
By \eqref{eq3.21} and \eqref{eq3.6} we get
$$ C \varepsilon \lambda_{\varepsilon}^{p}\log \frac{1}{\sqrt[p]{\varepsilon} \lambda_{\varepsilon}}\leq \Big(\varepsilon (\log \frac{1}{\varepsilon})\Big)^{\frac{(l-p^{*})}{(l-p)}}. $$
Clearly,
$$ \varepsilon^{\delta_{1}} \leq \sqrt[p]{\varepsilon} \lambda_{\varepsilon}
\leq\varepsilon^{\delta_{2}},$$
for some $  \delta_{1,2}\geq 0 $ and $  \varepsilon $ small enough, by \eqref{3.22} and \eqref{3.23}. It follows that
$$ \log \frac{1}{\sqrt[p]{\varepsilon} \lambda_{\varepsilon}}\sim \log \frac{1}{\varepsilon}.$$
Hence,
$$  \lambda_{\varepsilon}^{p} \lesssim \Big(\varepsilon (\log \frac{1}{\varepsilon})\Big)^{\frac{(l-p^{*})}{(l-p)}-1}=\Big(\varepsilon (\log \frac{1}{\varepsilon})\Big)^{-\frac{(p^{*}-p)}{(l-p)}}, $$
and
$$  \lambda_{\varepsilon} \lesssim \Big(\varepsilon (\log \frac{1}{\varepsilon})\Big)^{-\frac{(p^{*}-p)}{p(l-p)}},$$
and this concludes the proof.
\end{proof}

\subsection{Proofs}

The sharp upper estimates on $\lambda_\eps$ yield the following 
\begin{corollary}\label{3118} Let either $1<p<\sqrt {N}, $ or $N\geq 4$ and $\sqrt{N}\leq p<\frac{N+1}{2}$. Then
$$||v_{\varepsilon}||_{l}=O(1).$$
\end{corollary}
\noindent The boundedness of the $L^{l}$ norm also allows one to reverse the estimates of $||v_{\varepsilon}||_{p}$ via \eqref{eq3.21}.
\begin{corollary}  It holds that
$$ ||v_{\varepsilon}||^{p}_{p}=
\begin{cases}
O(1), & 1<p<\sqrt{N},\\
O(\log\frac{1}{\varepsilon}),& p=\sqrt{N},\; N\geq4,\\
O(\varepsilon^{\frac{(l-p^{*})[p-(p^{*}-p)(p-1)]}{(p^{*}-p)[(l-p^{*})(p-1)+p]}}), & \sqrt{N}<p<\frac{N+1}{2},\; N\geq4.
\end{cases}$$
\end{corollary}
\noindent We now prove that the $L^{l}$ bound implies an $L^{\infty}$ bound.
\begin{lemma}\label{13120}
Let either $1<p<\sqrt {N}, $ or $N\geq 4$ and $\sqrt{N}\leq p<\frac{N+1}{2}$.  It holds that 
\begin{equation}\label{03031}
||v_{\varepsilon}||_{\infty} =O(1).
\end{equation}
\end{lemma}

\begin{proof}
We start observing that by \eqref{3321} $v_{\varepsilon} $ is a positive solution to the inequality
$$-\Delta_{p} v_{\varepsilon} -V_{\varepsilon}(x) v_{\varepsilon}^{p-1} \leq 0, \quad x\in \mathbb{R}^{N},$$
with
$$V_{\varepsilon}(x):= S_{\varepsilon} v^{p^{*}-p}_{\varepsilon}(x).$$
By Lemma \ref{41} in the Appendix, we obtain
\begin{equation}\label{0330}
|v_{\varepsilon}(x)|\leq C_{l} ||v_{\varepsilon}||_{l} |x|^{-N/l} \quad x\neq0, 
\end{equation}
which combined with Corollary \ref{3118} yields
$$V_{\varepsilon}(x)\leq S_{\varepsilon} C^{p^{*}-p}_{l} ||v_{\varepsilon}||^{p^{*}-p}_{l}|x|^{-N(p^{*}-p)/l}\leq C_{*} |x|^{-pp^{*} /l},$$
for some uniform constant $ C_{*} >0$ independent on $\varepsilon $ or $x $. Hence, $v_{\varepsilon} $ is a positive solution to the inequality
\begin{equation} \label{0331}
-\Delta_{p} v_{\varepsilon} -V_{*}(x) v^{p-1}_{\varepsilon} \leq 0, \quad x\in \mathbb{R}^{N},
\end{equation}
with $ V_{*}(x)= C_{*} |x|^{-pp^{*}/l} \in L_{loc}^{s} (\mathbb{R}^{N})$ for some  $s>N/p$, since $l>p^{*}$.
With these preliminaries in place, one can invoke here the result on local boundedness Theorem $7.1.1 $ in \cite{PucciB}*{p.154} for subsolutions of \eqref{0331} to conclude. However, to make the proof selfcontained, we provide a simple argument to justify \eqref{03031}.

\noindent Integrating the inequality \eqref{0331} over a ball
$$\int\limits_{B_{|x|}(0)}-\Delta_{p} v_{\varepsilon}(x) dx \leq \int\limits_{B_{|x|}(0)}V_{*}(x) v^{p-1}_{\varepsilon}(x) dx, $$
and by the divergence theorem, taking into account the monotonicity of $ v_{\varepsilon}$ with respect to $|x| $ we have
\begin{eqnarray*}
\int\limits_{ B_{|x|}(0)}-\Delta_{p} v_{\varepsilon}(y) dy&=& \int\limits_{\partial B_{|x|}(0)}-|\nabla v_{\varepsilon}(\sigma)|^{p-2} \nabla v_{\varepsilon}(\sigma) \cdot \nu \,d\sigma=\\
&= &|\nabla v_{\varepsilon}(x)|^{p-1} \int\limits_{\partial B_{|x|}(0)}d\sigma=C_{1} |\nabla v_{\varepsilon}(x)|^{p-1} |x|^{N-1}.
\end{eqnarray*}
On the other hand
$$\int\limits_{B_{|x|}(0)} V_{*}(y) v^{p-1}_{\varepsilon}(y) dy= C_{2}\int\limits_{0}^{|x|} r^{-\frac{pp^{*}}{l}+N-1} v^{p-1}_{\varepsilon}(r)dr\leq$$
$$\leq C_{2}|v_{\varepsilon}(0)|^{p-1} \int\limits_{0}^{|x|} r^{-\frac{pp^{*}}{l}+N-1} dr= C_{3} |v_{\varepsilon}(0)|^{p-1} |x|^{-\frac{pp^{*}}{l}+N},$$
since $ -\frac{pp^{*}}{l}+N >0 $.
Hence
\begin{equation} \label{0332}
|\nabla v_{\varepsilon}(x)|^{p-1}\leq \frac{C_{4}}{|x|^{N-1}}\int\limits_{B_{|x|}(0)} V_{*}(r) v^{p-1}_{\varepsilon}(r) dr \leq C_{5} |v_{\varepsilon}(0)|^{p-1} |x|^{1-pp^{*}/l},
\end{equation}
for some $C_{4},C_{5}>0$ independent of $\varepsilon $ and $x $. Integrating again from $0$ to $x_{0} $ after writing \eqref{0332} in this form
$$-\frac{d}{dr} v_{\varepsilon}(|x|)\leq  C_{6} |v_{\varepsilon}(0)| |x|^{(1-pp^{*}/l)/(p-1)},$$
we have
\begin{equation}\label{03034}
v_{\varepsilon}(0)\leq v_{\varepsilon}(x_{0}) +C_{7} v_{\varepsilon} (0) |x_{0}|^{\frac{p(l-p^{*})}{l(p-1)}},
\end{equation}
for some $ C_{7} $ independent of $\varepsilon $ and $x $.
We pick $ A$ small enough such that for all $|x_{0}|\leq A $ we have
$$v_{\varepsilon}(0)(1-C_{7}A^{\frac{p(l-p^{*})}{l(p-1)}}) \leq v_{\varepsilon}(0)(1-C_{7}|x_{0}|^{\frac{p(l-p^{*})}{l(p-1)}})\leq v_{\varepsilon}(x_{0}).$$
Then
$$C_{8}v_{\varepsilon}(0)\leq v_{\varepsilon}(x_{0}), \quad \text{for\; all }\; x_{0}, \; |x_{0}|< A,$$
where $C_{8}= 1-C_{7}A^{\frac{p(l-p^{*})}{l(p-1)}}$.
Hence by taking the power $ l$ and  integrating we obtain
$$\int\limits_{|x|<A}C_{9}|v_{\varepsilon}(0)|^{l} dx\leq \int\limits_{|x|<A}|v_{\varepsilon}(x)|^{l}dx.$$
which by Corollary \ref{3118} immediately concludes the proof.
\end{proof}
\noindent 
By elliptic estimates for the $p$-Laplacian, we have the following

\begin{corollary}\label{3121}
Let either $1<p<\sqrt {N}, $ or $N\geq 4$ and $\sqrt{N}\leq p<\frac{N+1}{2}$. It holds that
$v_{\varepsilon}\rightarrow W_{1}$ in $C_{loc}^{1,\alpha}(\mathbb{R}^{N}) $ and $L^{s} (\mathbb{R}^{N})$ for any $ s\geq p^{*} $. In particular,
$$ v_{\varepsilon}(0) \simeq U_{1}(0). $$
\end{corollary}
\begin{proof}
As a consequence of the $L^{\infty} $ bound of Lemma \ref{13120} and the convergence of $v_{\varepsilon} $ to the Sobolev minimiser $W_{1}$ in $D^{1,p}(\mathbb{R}^{N})$ via the compactness result  in Lemma \ref{41} we obtain the convergence in $L^{s}(\mathbb{R}^{N}) $ for any $s\geq p^{*} $.
\noindent Since we can write \eqref{3321} in the form
$$ -\Delta_{p}v_{\varepsilon}=g_\eps(v_{\varepsilon}),$$
and by Lemma \ref{13120} we have
$$ ||g_\eps(v_{\varepsilon})||_{L_{loc}^{\infty}(\mathbb{R}^{N})}< C,$$
uniformly with respect to $\varepsilon, $ then by \cite{DiBenedetto}*{Theorem $2$} we have
$$||v_{\varepsilon} ||_{C_{loc}^{1,\alpha}(\mathbb{R}^{N})}<C,$$ uniformly with respect to $\varepsilon $. It follows that by the classical Arzel\'a-Ascoli theorem that for a suitable sequence $\varepsilon\rightarrow 0 $ we have
$$v_{\varepsilon} \rightarrow W_{1} \quad \text{in} \; C_{loc}^{1,\alpha'}(\mathbb{R}^{N}),$$
where $\alpha<\alpha' $.
\end{proof}
\begin{proof}[Proof of Theorem \ref{s123}]    
   The proof follows immediately from Lemma \ref{2.3.6} and Lemma \ref{2.3.7}, 
 	which yield the upper and lower estimates on $\lambda_{\varepsilon}$.
\end{proof}

\begin{proof} [Proof of Theorem \ref{main-cor}] The proof follows from the sharp upper bound on $\lambda_{\varepsilon}$ in Corollaries \ref{31015}-\ref{31017}, and from Corollary \ref{3121}. In particular since from Corollary \ref{3121} and in view of (\ref{p27}) we have
$$u_{\varepsilon}(0) \sim \lambda^{-\frac{N-p}{p}}_{\varepsilon}v_{\varepsilon}(0),$$
then by the sharp estimate of $\lambda_{\varepsilon} $ we have the exact rate of the groundstate $u_{\varepsilon} (0)$ in the present critical case
\begin{equation}
u_{\varepsilon}(0)\sim
\begin{cases} 
\varepsilon^{\frac{l}{(l-p)}}& 1<p<\sqrt{N},\; N\ge 2,\\
\varepsilon^{\frac{N-p} {p[(l-p^{*})(p-1)+p]}} &\sqrt{N}<p<\frac{N+1}{2},\; N\geq4,\\
(\varepsilon (\log \frac{1}{\varepsilon})\Big)^{\frac{l}{(l-p)}} & p=\sqrt{N},\;  N\geq4.
\end{cases}
\end{equation}
\end{proof}
   
\section{Proof of Theorem \ref{ss122}: supercritical case $q>p^{*}$}  \label{7} 

In this section, we consider the supercritical case $q>p^*$ and prove Theorem \ref{ss122} stated in the Introduction,
which essentially says that for $q>p^{*}$ groundstate solutions $ u_{\varepsilon}$ converge as $\varepsilon\rightarrow 0 $ to a non-trivial radial groundstate solution  of the formal limit equation \eqref{29}.
This result extends \cite{Moroz-Muratov}*{Theorem $ 2.3$} to $p\neq2$.

\subsection{The limiting PDE}
From the results of Section \ref{PDEsection} we know that for $q>p^{*} $ the limit equation
\begin{equation}\tag{$P_{0}$}\label{29}
-\Delta_p u-u^{q-1}+u^{l-1}=0\quad\text{in } \;\mathbb{R}^N,
\end{equation}
admits positive radial groundstates solutions $ u_{0} \in D^{1,p}(\mathbb{R}^{N})\cap L^{l}(\mathbb{R}^{N})$,
which are, since they are radial, fast decaying, namely such that
\begin{equation}\label{030308}
u_{0}(x)\sim |x|^{-\frac{N-p}{p-1}} \quad \text{as}\quad |x| \rightarrow \infty.
\end{equation}
Note that by construction $u_{0}\in C_{loc}^{1,\alpha}(\mathbb{R}^{N})$.
Moreover $u_{0} $ admits a variational charachterization in the Sobolev space $D^{1,p} (\mathbb{R}^{N})$ via the rescaling
$$u_{0}(x):= w_{0}\Big(\frac{x}{\sqrt[p]{S_{0}}}\Big),$$
where $ w_{0}$ is a positive radial minimizer of the constrained minimization problem
\begin{equation}\tag{$S_{0}$}\label{991}
 S_{0}:= \inf \Big\{\int\limits_{\mathbb{R}^{N}}|\nabla w|^{p}dx\Big| w\in D^{1,p}(\mathbb{R}^{N}), \quad p^{*}\int\limits_{\mathbb{R}^{N}} \tilde F_{0}(w)dx=1 \Big\},
 \end{equation}
where
$$ \tilde F_{0}(w)= \int_{0}^{w} \tilde f_{0}(s) ds,$$
and $\tilde f_0(s)$ is a truncation of the nonlinearity
$$f_0(s)=|s|^{q-2} s - |s|^{l-2} s,$$
as described in Section \ref{PDEsection}.
Then the minimization problem \eqref{991} is well defined on $D^{1,p}(\mathbb{R}^{N}) $.
The minimizer $w_{0}$ satisfies the Euler-Lagrange equation
$$ - \Delta_{p}w_{0}= S_{0} (w_{0}^{q-1} -w_{0}^{l-1}).$$
Moreover, $w_{0} $ satisfies Nehari's identity
$$\int\limits_{\mathbb{R}^{N}}|\nabla w_{0}|^{p}dx =S_{0} \Big( \int\limits_{\mathbb{R}^{N}}|w_{0}|^{q}dx- \int\limits_{\mathbb{R}^{N}}|w_{0}|^{l}dx\Big), $$
which yields 

\begin{equation}\label{030390}
1 = ||w_{0}||_{q}^{q}- ||w_{0}||_{l}^{l}.
\end{equation}
From the Poho\v{z}aev identity
$$\int\limits_{\mathbb{R}^{N}}|\nabla w_{0}|^{p}dx =S_{0}  p^{*}\Big(\frac{1}{q} \int\limits_{\mathbb{R}^{N}}|w_{0}|^{q}dx-\frac{1}{l} \int\limits_{\mathbb{R}^{N}}|w_{0}|^{l}dx\Big),$$
we have

\begin{equation}\label{030400}
1 = \frac{p^{*}}{q} ||w_{0}||_{q}^{q}-\frac{p^{*}}{l} ||w_{0}||_{l}^{l}.
\end{equation}
Hence from \eqref{030390} and \eqref{030400} we obtain the relation
$$ ||w_{0}||_{q} ^{q}-||w_{0} ||^{l}_{l}=\frac{p^{*}}{q}||w_{0}||^{q}_{q}-\frac{p^{*}}{l}||w_{0}||^{l}_{l}=1,$$
from which we obtain the expressions
$$||w_{0}||^{q}_{q}= \frac{q(l-p^{*})}{p^{*}(l-q)}, \quad ||w_{0}||^{l}_{l}=\frac{l(q-p^{*})}{p^{*}(l-q)}.$$

\subsection{Energy estimates and groundstate asymptotics}
The relations between $ \mathcal{S}_{\varepsilon}$ and $\mathcal{S}_{0}$ is provided by introducing the convenient scaling-invariant quotient
\begin{equation}\label{S0Rw}
\mathcal{S}_{0}(w):= \frac{\int\limits_{\mathbb{R}^{N}}|\nabla w|^{p}dx }{\Big(p^{*}\int\limits_{\mathbb{R}^{N}}\tilde F_{0}(w)dx\Big)^{(N-p)/N} }, \quad w\in \mathcal{M}_{0},
\end{equation}
where
$$  \mathcal{M}_{0}:= \Big\{w \in D^{1,p}(\mathbb{R}^{N}),\; \int\limits_{\mathbb{R}^{N}} \tilde F_0(w)dx>0\Big\}.$$ Note that, by a rescaling argument, this is equivalent to \eqref{991} :
$$ S_{0}=\displaystyle\inf\limits_{w\in\mathcal{M}_{0}}\mathcal{S}_{0}(w).$$
\begin{lemma}\label{24}
	For all $1<p<N$, it holds that
	$$0<S_{\varepsilon} - S_{0}\rightarrow0, \quad \text{as} \quad \varepsilon\rightarrow 0.$$
\end{lemma}
\begin{proof}
	To show that $\mathcal{S}_{0}<\mathcal{S}_{\varepsilon}$, simply note that
	\begin{equation}
S_{0}\leq \mathcal{S}_{0}(w_{\varepsilon})< \mathcal{S}_{\varepsilon}(w_{\varepsilon})=S_{\varepsilon}.
	\end{equation}
	To estimate $ S_{\varepsilon}$ from above we test \eqref{e2.4} with the minimizer $ w_{0}.$  By \eqref{030308}, we have $w_{0}\in L^{p}(\mathbb{R}^{N})$ if and only if $1<p<\sqrt{N}$. We break the proof by analysing the higher and lower dimensions separately.
	
	\noindent
	{\em Case $1<p<\sqrt{N}$.}
	Using $w_0$ as a test function for \eqref{e2.4}, we obtain
	\begin{eqnarray}
		S_{\varepsilon}\leq\mathcal{S}_{\varepsilon}(w_{0})
		\leq \frac{S_{0}}{(1-\frac{ \varepsilon p^{*} }{p}||w_{0}||_{L^{p}(\mathbb{R}^{N})}^{p})^\frac{N-p}{N}}\leq S_{0}+O(\varepsilon),
	\end{eqnarray}
which proves the statement for $1<p<\sqrt{N}$.

In the cases $p=\sqrt{N} $ and $ \sqrt{N} <p<N $, given $R>1$
we pick a cut-off function $ \eta_{R} \in C^{\infty}_{0}(\mathbb{R})$ such that $ \eta_{R}(r)=1 $ for $ |r|<R $, $ 0<\eta_{R}<1 $ for $ R<|r|<2R $, $ \eta_{R}=0 $ for $ |r|>2R $ and $ |\eta^{'}_{R}|\leq 2/R $. By \eqref{030308}, for $s>\frac{N}{N-p}$ we obtain
	$$
	\int_{\mathbb{R}^{N}} |\nabla(\eta_{R} w_{0}(x)|^{p}dx=S_{0}+O(R^{-\frac{N-p}{p-1}} ),
	$$
	$$
	\int_{\mathbb{R}^{N}} |\eta_{R} w_{0}(x)|^{s}dx =|| w_{0}||^{s}_{L^{s}(\mathbb{R}^{N})}\Big(1- O(R^{-\frac{(N-p)s}{p-1}+N})\Big),
	$$
	$$
	\int_{\mathbb{R}^{N}} |\eta_{R} w_{0}(x)|^{p}dx=\begin{cases}
	O( \log(R)), & p=\sqrt{N},\\
	O (R^{\frac{p^{2}-N}{p-1}}), & \sqrt{N}<p<N.\\
	\end{cases}
	$$
	
		\noindent
	{\em
		Case $p=\sqrt{N}$.} Let $R=\varepsilon^{-1}$. Testing \eqref{e2.4} with $ \eta_{R}w_{0}$ and since $ q>p^{*}$, we get
	\begin{multline*}
	$$ S_{\varepsilon}\leq\mathcal{S}_{\varepsilon}(\eta_{R}w_{0})\leq
	\Big(S_{0}+O(R^{-\frac{N-p}{p-1}} ) \Big)\div\\
	\Big(\frac{p^{*}}{q}|| w_{0}||^{q}_{q}\big(1- O(R^{-\frac{(N-p)q}{p-1}+N})\big)-\frac{ \varepsilon p^{*}}{p}O( \log R)-\frac{p^{*}}{l} || w_{0}||^{l}_{l}\big(1- O(R^{-\frac{(N-p)l}{p-1}+N})\big)\Big)^\frac{N-p}{N} $$
	\end{multline*}
	$$=\frac{S_{0}+O(\varepsilon^{\frac{N-p}{p-1}} ) }{\Big(1- o(\varepsilon^{\frac{N}{p-1}})-  O(\varepsilon \log\frac{1}{\varepsilon})\Big)^{\frac{N-p}{N}}}	\leq S_{0}+O\Big(\varepsilon \log\frac{1}{\varepsilon} \Big),$$
	from which the claim follows.
	
	\noindent
	{\em Case $\sqrt{N}<p<N$.}
	Let $R=\varepsilon^{-\frac{1}{p}}$. We test \eqref{e2.4} with $ \eta_{R}w_{0}$ and as $ q>p^{*}$, we obtain
	\begin{multline*}
	$$ S_{\varepsilon}\leq\mathcal{S}_{\varepsilon}(\eta_{R}w_{0})\leq
	\Big(S_{0}+O((R)^{-\frac{N-p}{p-1}} ) \Big)\div\\
	\Big(\frac{p^{*}}{q}|| w_{0}||^{q}_{q}(1- O(R^{-\frac{(N-p)q}{p-1}+N}))-\frac{ \varepsilon p^{*}}{p}O (R^{\frac{p^{2}-N}{p-1}})-\frac{p^{*}}{l} || w_{0}||^{l}_{l}(1- O((R)^{-\frac{(N-p)l}{p-1}+N}))\Big)^\frac{N-p}{N} $$
	\end{multline*}
	$$\leq \frac{S_{0}+O(\varepsilon^{\frac{N-p}{p(p-1)}} ) }{\Big(1- o(\varepsilon^{\frac{N}{p(p-1)}})-   O (\varepsilon^{-\frac{p^{2}-N}{p(p-1)}+1})\Big)^{\frac{N-p}{N}}}\leq S_{0}+O(\varepsilon^{\frac{N-p}{p(p-1)}}),$$
which completes the proof.
\end{proof}

\begin{lemma}\label{325}
	It holds that $||w_{\varepsilon}||_{\infty}\leq1$ and $||w_{\varepsilon}||_{s}\lesssim1$ for all $s>p^{*} $.
\end{lemma}
\begin{proof}
	Note that by \eqref{p27} we have
	$$
	||w_{\varepsilon}||_{\infty} =||u_{\varepsilon}||_{\infty}\leq 1.
	$$
	By Sobolev's inequality and Lemma \ref{24} we have
	$$  ||w_{\varepsilon}||^{p}_{p^{*}}\leq S^{-1}_{*}||\nabla w_{\varepsilon}||^{p}_{p}=S^{-1}_{*} S_{\varepsilon}=S^{-1}_{*} S_{0}\big(1+o(1)\big).$$
	Hence for every $s > p^{*}$,
	$$||w_{\varepsilon}||^{s}_{s}\leq ||w_{\varepsilon}||_{p^{*}}^{p^{*}},$$
	which concludes the proof.
\end{proof}

\begin{lemma}\label{326}
	For all $1<p<N$, we have
	$$\varepsilon ||w_{\varepsilon}||^{p}_{p}\rightarrow 0\quad\text{as} \quad\eps\to 0.$$
\end{lemma}
\begin{proof}
	Observing that $w_{\varepsilon} $ is an optimiser to \eqref{e2.4}, it follows that
	\begin{equation}\label{33}
	1= p^{*} \int\limits_{\mathbb{R}^{N}}F_{\varepsilon}(w_{\varepsilon}) dx = p^{*} \int\limits_{\mathbb{R}^{N}}F_{0}(w_{\varepsilon})-p^{*}\frac{\varepsilon}{p}||w_{\varepsilon} ||^{p}_{p} .
	\end{equation}
	Hence
	$$ \mathcal{S}_{0}(w_{\varepsilon})= \frac{\int\limits_{\mathbb{R^{N}}}|\nabla w_{\varepsilon}|^{p}dx }{\Big(p^{*}\int\limits_{\mathbb{R^{N}}}F_{0}(w_{\varepsilon}) dx\Big)^{(N-p)/N} }= \frac{S_{\varepsilon} }{\Big(1+\frac{p^{*}}{p} \varepsilon ||w_{\varepsilon}||_{p}^{p}\Big)^{(N-p)/N} }.$$
If by contradiction we had $\limsup_{\varepsilon\rightarrow 0} \varepsilon||w_{\varepsilon}||^{p}_{p}=m>0,$
	then by Lemma \ref{24} for any sequence $ \varepsilon_{n}\rightarrow 0$, we would obtain
	$$S_{0}\leq \mathcal{S}_{0}(w_{\varepsilon_{n}})= \frac{S_{\varepsilon_{n}} }{\Big(1+\frac{p^{*}}{p} \varepsilon_{n} ||w_{\varepsilon_{n}}||_{p}^{p}\Big)^{(N-p)/N} }\leq\frac{S_{0}\big(1+o(1)\big) }{1+\frac{p^{*}}{p} m}<S_{0},$$
	and this, as it is clearly a contradiction, concludes the proof.
\end{proof}

 \begin{theorem}\label{246}
Let $1<p<N$ and $q> p^{*} $. As $\varepsilon\rightarrow 0, $
	the family of groundstates $ u_{\varepsilon}$ converges to a groundstate $ u_{0}$ in $D^{1,p}(\mathbb{R}^{N}) $, $L^{l} (\mathbb{R}^{N})$ and $ C_{loc}^{1,\alpha }(\mathbb{R}^{N})$ to \eqref{2s}. In particular
	$$u_{\varepsilon}(0) \simeq u_{0}(0).$$ Furthermore $u_0$ is fast decaying, namely \begin{equation*}
u_{0}(x)\sim |x|^{-\frac{N-p}{p-1}} \quad \text{as}\quad |x| \rightarrow \infty.
\end{equation*}	
	\end{theorem}
\begin{proof}
	Since the family $w_{\varepsilon}$ is bounded in $D^{1,p}(\mathbb{R}^{N}) $ then there exists a subsequence ${w_{\varepsilon_{n}}} $ such that
	$$w_{\varepsilon_{n}}\rightharpoonup \tilde{w} \quad  \text{in} \quad D^{1,p}(\mathbb{R}^{N})\quad  \text{and}  \quad w_{\varepsilon_{n}}\rightarrow \tilde{w} \quad \text{a.e \; in } \; \mathbb{R}^{N}, \quad \text{as}  \; n\rightarrow \infty $$
	where $ \tilde{w} \in D^{1,p}(\mathbb{R}^{N})$ is a radial function. By Sobolev's inequality,
	the sequence $(w_{\varepsilon_{n}} )$ is bounded in $ L^{p^{*}}(\mathbb{R}^{N})$. Using Lemma \ref{41} we conclude that
	$$
	w_{\varepsilon_{n}}\rightarrow \tilde{w} \quad \text{in} \quad L^{s}(\mathbb{R}^{N}\setminus B_{r}(0)) \quad \text{for}\; r>0 \;\text{and}\; s\in (p^{*},\infty).$$
	\noindent Taking into account Lemma \ref{326} and \eqref{33} we also obtain
	$$
	\int\limits_{\mathbb{R}^{N}}F_{0}(\tilde{w}) dx = \lim\limits_{n\rightarrow \infty} \int\limits_{\mathbb{R}^{N}}F_{0}(w_{\varepsilon_{n}})dx=\lim\limits_{n\rightarrow \infty}(1+p^{*}\frac{\varepsilon_{n}}{p}||w_{\varepsilon_{n}} ||^{p}_{p})=1 .$$
	By the weak lower semicontinuity property of the norm we also have that
	$$||\nabla \tilde{w}||_{p}^{p}\leq \liminf_{n\longrightarrow \infty }||\nabla w_{\varepsilon_{n}}||_{p}^{p}=S_{0},$$
	i.e. $\tilde{w} $ is a minimizer for \eqref{991}.
	We now claim that
	\begin{equation} \label{339}
	\nabla w_{\varepsilon_{n}} \rightarrow \nabla \tilde{w} \quad a.e.\quad  \text{on}\; \mathbb{R}^{N},
	\end{equation}
	and then by Brezis-Lieb Lemma \cite{Brezis}, $ (w_{\varepsilon_{n}})$ converges strongly to $\tilde{w} $ in $ D^{1,p}(\mathbb{R}^{N})$. In fact, arguing as in \cite{Mercuri}*{Theorem $3.3$}
	 (see also \cite{MercuriSq}*{Proposition $2.3$}, define a bounded function	
	 $$T:=\begin{cases}
	s \quad&  \text{if} \quad |s|\leq 1,\\
	\frac{s}{|s|}&\text{if} \quad |s|> 1,
	\end{cases} $$
	and consider a sequence $(B_{k})$ of open subsets of $\mathbb{R}^{N} $ such that $\bigcup\limits_{k=1}^{\infty}B_{k}= \mathbb{R}^{N} $.
	Then if
	\begin{equation}\label{933}
	\lim\limits_{n\rightarrow\infty} \int\limits_{B_{k}}(|\nabla w_{\varepsilon_{n}}|^{p-2}\nabla w_{\varepsilon_{n}}-|\nabla \tilde{w}|^{p-2}\nabla \tilde{w})\cdot \nabla T (w_{\varepsilon_{n}}-\tilde{w}) dx\rightarrow 0,
	\end{equation}
	for every $ k$, then
	$$\nabla w_{\varepsilon_{n}} \rightarrow \nabla \tilde{w} \quad a.e.\quad  \text{on}\;B_{k},$$ and hence by a Cantor diagonal argument, \eqref{339} is satisfied.
	
	\noindent To show \eqref{933}, we introduce a cut-off function
	$$\rho(x):=\begin{cases}
	1 \quad&  \text{if} \quad |x|\leq k,\\
	0 &\text{if} \quad |x|\geq k+1,
	\end{cases} $$
	and since $$\big(|\nabla w_{\varepsilon_{n}}|^{p-2}\nabla w_{\varepsilon_{n}}-|\nabla \tilde{w}|^{p-2}\nabla \tilde{w}\big) \nabla T (w_{\varepsilon_{n}}-\tilde{w}) \geq 0, $$
	then
	\begin{eqnarray*}
		0&\leq& 	 \int\limits_{B_{k}}\big(|\nabla w_{\varepsilon_{n}}|^{p-2}\nabla w_{\varepsilon_{n}}-|\nabla \tilde{w}|^{p-2}\nabla \tilde{w}\big) \nabla T (w_{\varepsilon_{n}}-\tilde{w}) dx  \\
		& \leq &	 \int\limits_{B_{k+1}}\Big[\big(|\nabla w_{\varepsilon_{n}}|^{p-2}\nabla w_{\varepsilon_{n}}-|\nabla \tilde{w}|^{p-2}\nabla \tilde{w}\big) \nabla T (w_{\varepsilon_{n}}-\tilde{w})\Big] \rho(x) dx  \\
		& \leq &	 \Big|\int\limits_{B_{k+1}}\big(|\nabla w_{\varepsilon_{n}}|^{p-2}\nabla w_{\varepsilon_{n}}-|\nabla \tilde{w}|^{p-2}\nabla \tilde{w}\big) \nabla \big(\rho T (w_{\varepsilon_{n}}-\tilde{w})\big) dx \Big|  \\
		& +	& \Big|\int\limits_{B_{k+1}}\big(|\nabla w_{\varepsilon_{n}}|^{p-2}\nabla w_{\varepsilon_{n}}-|\nabla \tilde{w}|^{p-2}\nabla \tilde{w}\big)  T (w_{\varepsilon_{n}}-\tilde{w}) \nabla\rho dx \Big| \rightarrow 0,
	\end{eqnarray*}
	as $n\rightarrow\infty.$
In fact
	$$  \Big|\int\limits_{B_{k+1}}\big(|\nabla w_{\varepsilon_{n}}|^{p-2}\nabla w_{\varepsilon_{n}}-|\nabla w_{0}|^{p-2}\nabla \tilde{w}\big) \nabla \big(\rho T (w_{\varepsilon_{n}}-\tilde{w})\big) dx \Big| $$
	$$ \leq \Big|\int\limits_{\mathbb{R}^{N}}|\nabla w_{\varepsilon_{n}}|^{p-2}\nabla w_{\varepsilon_{n}}\nabla \big(\rho T (w_{\varepsilon_{n}}-\tilde{w})\big) dx\Big| +\Big|\int\limits_{\mathbb{R}^{N}}|\nabla \tilde{w}|^{p-2}\nabla \tilde{w} \nabla \big(\rho T (w_{\varepsilon_{n}}-\tilde{w})\big) dx \Big|$$
	$$ = \Big|\int\limits_{\mathbb{R}^{N}}f_{\varepsilon}(w_{\varepsilon_{n}}) \rho T (w_{\varepsilon_{n}}-\tilde{w}) dx\Big| +\Big|\int\limits_{\mathbb{R}^{N}}f(\tilde{w}) \rho T (w_{\varepsilon_{n}}-\tilde{w}) dx \Big|\rightarrow 0,$$
	by local compactness.
	Moreover, by H\"{o}lder's inequality  and since $T$ is bounded and $w_{\varepsilon_{n}}-\tilde{w}\rightarrow 0$ a.e. on $\mathbb{R}^{N} $, then by dominated convergence theorem, we have
	\begin{multline*}
	$$\Big|\int\limits_{B_{k+1}}(|\nabla w_{\varepsilon_{n}}|^{p-2}\nabla w_{\varepsilon_{n}}-|\nabla \tilde{w}|^{p-2}\nabla \tilde{w})T (w_{\varepsilon_{n}}-\tilde{w})\nabla \rho dx \Big|  \\
	\leq C \Big(\int\limits_{B_{k+1}} |T (w_{\varepsilon_{n}}-\tilde{w})|^{p} |\nabla \rho|^{p} dx \Big)^{\frac{1}{p}} \rightarrow 0,$$
	\end{multline*}
	and hence \eqref{933} follows.
	As a consequence $ (w_{\varepsilon_{n}})$ converges to $ \tilde{w}$ in $ D^{1,p}(\mathbb{R}^{N})$ and in $ L^{s}(\mathbb{R}^{N})$ for any $ s\geq p^{*}$, where $ \tilde{w}$ is a minimizer of \eqref{991} satisfying the constraint.  Similarly to the proof of Corollary \ref{3121}, using Lemma \ref{325}, by uniform elliptic estimates we conclude that $(w_{\varepsilon_{n}})$ converges to $\tilde{w}_{0}$ in $C_{loc}^{1,\alpha}(\mathbb{R}^{N})$. The decay follows from Lemma \ref{decaylemma}. This concludes the proof. \end{proof}
\begin{proof}[Proof of Theorem \ref{ss122}] The statement follows directly from Theorem \ref{246} and Lemma \ref{326}.
\end{proof}

\section{Proof of Theorem \ref{s121}: subcritical case $p<q<p^{*}$   }\label{5}

In this section, we consider the subcritical case $p<q<p^*$ and prove Theorem \ref{s121} showing that, after the canonical rescaling \eqref{03s}, the groundstate solutions $ u_{\varepsilon}$ converge as $\varepsilon\rightarrow 0 $ to the unique non-trivial radial groundstate solution  to the limit equation  \eqref{4s}.
This result extends \cite{Moroz-Muratov}*{Theorem $ 2.1$} to $p\neq2$.

\noindent Since by Poho\v{z}aev's identity the equation \eqref{29} has no positive finite energy solutions, to understand the asymptotic behaviour of the groundstates $u_{\varepsilon}$ we consider the rescaling in \eqref{03s}, which transforms \eqref{1s} into \eqref{3s}, whose limit problem as $\varepsilon \rightarrow 0$ is \eqref{4s}.\\
\noindent Pick $G_{\varepsilon}:\mathbb{R}\rightarrow\mathbb{R}, $ a bounded truncated function such that
\begin{equation*}
G_{\varepsilon}(w)=
\frac{1}{q}|w|^{q}-\frac{1}{p}|w|^{p} -\frac{\varepsilon^{\frac{l-q}{q-p}}}{l}|w|^{l},
\end{equation*}
for $ 0<w\leq \varepsilon^{-\frac{1}{q-p}}    $, $G_{\varepsilon}(w)\leq0$ for $w>\varepsilon^{-\frac{1}{q-p}}$ and $G_{\varepsilon}(w)=0$ for $w\leq0$.
For $ \varepsilon \in [0,\varepsilon^{*}) $, we set
\begin{equation}\tag{$S'_{\varepsilon}$}\label{5p}
 S'_{\varepsilon}:= \inf \Big\{\int\limits_{\mathbb{R}^{N}}|\nabla w|^{p}dx\;\Big| \; w\in W^{1,p}(\mathbb{R}^{N}), \quad p^{*}\int\limits_{\mathbb{R}^{N}} G_{\varepsilon}(w)dx=1 \Big\},
 \end{equation}
a well-defined family of constrained minimisation problems, which share, together with the limit problem $ (S'_{0})$, the same functional setting $W^{1,p}(\mathbb{R}^{N}) $.
By Theorem \ref{th3.2.9}, \eqref{5p} possesses a radial positive minimizer $ w_{\varepsilon}$ for every $ \varepsilon \in [0,\varepsilon^{*}) $.
The rescaled function
$$v_{\varepsilon}(x):=w_{\varepsilon}\big(\frac{x}{\sqrt[p]{S'_{\varepsilon}}}\big), $$
is a radial groundstate of  \eqref{3s}.

\noindent We estimate \eqref{5p} by means of the dilation invariant representation
$$ \mathcal{S}'_{\varepsilon}(w):= \frac{\int\limits_{\mathbb{R^{N}}}|\nabla w|^{p}dx }{\Big(p^{*}\int\limits_{\mathbb{R^{N}}}G_{\varepsilon}(w) dx\Big)^{(N-p)/N} }, \quad w\in \mathcal{M}'_{\varepsilon} , $$
where $\mathcal{M}'_{\varepsilon}:=\{0 \leq w \in W^{1,p}(\mathbb{R}^{N}),\int\limits_{\mathbb{R^{N}}}G_{\varepsilon}(w) dx>0\}.$ We have
$$S'_{\varepsilon} =\displaystyle\inf_{w\in \mathcal{M}'_{\varepsilon} } \mathcal{S}'_{\varepsilon}(w), $$
and for $\varepsilon$ small enough we have
\begin{equation}
S'_{0}\leq \mathcal{S}'_{0}(w_{\varepsilon})< \mathcal{S}'_{\varepsilon}(w_{\varepsilon})=S'_{\varepsilon}.
\end{equation}
This follows by observing that as $p^{*}\int\limits_{\mathbb{R}^{N}}G_{\varepsilon}(w_{\varepsilon}) dx=1$ and $ G_{\varepsilon}(s)$ is a decreasing function of $\varepsilon $ for each $ s>0$, we have $w_{\varepsilon}\in \mathcal{M}'_{0} $, and the second inequality follows again by monotonicity. Observe that by continuity $w_{0}\in \mathcal{M}'_{\varepsilon} $ for sufficiently small $\varepsilon $. As a consequence, by testing\eqref{5p} with $w_0$, that for $\varepsilon $ small enough, we have that 
$$ S'_{\varepsilon}\leq \mathcal{S}'_{\varepsilon}(w_{0})= \frac{S'_{0} }{\Big(1-\frac{p^{*}}{l}\varepsilon^{\frac{l-q}{q-p}}||w_{0}||_{l}^{l}\Big)^{(N-p)/N} }\leq S'_{0}+O(\varepsilon^{\frac{l-q}{q-p}}). $$
Hence $S'_{\varepsilon} \rightarrow S'_{0} $.
\noindent Reasoning as in Lemma \ref{314} , we obtain that
$$ ||w_{\varepsilon}||^{q}_{q}=\frac{(l-p^{*})q}{(l-q)p^{*}}+\frac{q(l-p)}{p(l-q)}||w_{\varepsilon}||^{p}_{p}.$$
Inserting this identity into the definition of $S'_{0}(w_{\varepsilon})$ and using the convergence of $S'_{\varepsilon}$ to $S'_{0}$, one can easily check that

\begin{equation}\label{03055}
\lim\limits_{\varepsilon\rightarrow 0}||w_{\varepsilon}||_{p}^{p}= \frac{p(p^{*}-q)}{p^{*}(q-p)}, \quad \lim\limits_{\varepsilon\rightarrow 0} ||w_{\varepsilon} ||_{q}^{q}= \frac{q(p^{*}-p)}{p^{*}(q-p)}.
\end{equation}
Therefore $p^{*}\int\limits_{\mathbb{R}^{N}}G_{0}(w_{\varepsilon}) dx\rightarrow 1 \quad \text{as} \quad \varepsilon\rightarrow 0. $ We have then achieved that a rescaling $\lambda_{\varepsilon}\rightarrow1$ exists such that $p^{*}\int\limits_{\mathbb{R}^{N}}G_{0}(\tilde{w}_{\varepsilon}) dx= 1$ and $ \mathcal{S}'_{\varepsilon}(\tilde{w}_{\varepsilon}) \rightarrow S'_{0}$ with
$\tilde{w}_{\varepsilon}(x):=w_{\varepsilon}(\lambda_{\varepsilon}x). $
It follows that $(\tilde{w}_{\varepsilon}) $ is a minimizing one parameter family for $(S'_{0})$ that satisfies the constraint used in the method which yields Theorem \ref{th3.2.9}. By Theorem \ref{th3.2.9} we conclude that for a suitable sequence $\varepsilon_{n}\rightarrow 0 $, it holds $\tilde{w}_{\varepsilon_{n}}\rightarrow \tilde{w}$ strongly in $W^{1,p}(\mathbb{R}^{N})$, and since $\lambda_{\varepsilon}\rightarrow 1,$ it holds that $w_{\varepsilon_{n}}\rightarrow \tilde{w}$, where $ \tilde{w}$ is the minimizer of $ (S'_{0})$ satisfying the constraint. By the uniqueness  of minimizer of \eqref{4s}, we have $ \tilde{w}=w_{0}.$ 
An obvious modification of the proof of Lemma \ref{13120}, using $||w_{\varepsilon}||_{p^{*}},$ yields that $||w_{\varepsilon}||_{\infty}\lesssim 1$ as $\varepsilon \rightarrow 0$. By uniform elliptic estimates we conclude that $ w_{\varepsilon}$ converges to $w_{0} $ in $ L^{s}(\mathbb{R}^{N})$ for any $ s\geq p$ and in $C^{1,\alpha}_{loc} (\mathbb{R}^{N})$, and therefore
the proof of Theorem \ref{s121} is complete.

\appendix

\section{Radial functions} \label{A.3}

We recall that for $ u \in L^{1}( \mathbb{R}^{N} ) $, the radially decreasing rearrangement of a function $ u $ is denoted by $ u^{*} $ and it is such that for any $ \alpha > 0 $ it holds that
\begin{equation*}
 \big | x\in \R^N:  u(x)^{*}\geq \alpha \big| =  \big| x\in \R^N: |u(x)| \geq \alpha \big|,
\end{equation*}
where $ \big|\cdot \big| $ denotes the Lebesgue measure in $\R^N$. We recall that
\begin{equation*}
\int_{\mathbb{R}^{N}}F(u) dx = \int_{\mathbb{R}^{N}}F(u^{*}) dx,
\end{equation*}
for every continuous $F$ such that $F(u) $ is summable. \newline
The following fundamental properties of rearrangements can be found e.g. in \cite{Willem b}:

\begin{lemma}
Let $ 1 \leq p < \infty $ and $ u,v \in L^{p}(\mathbb{R}^{N}) $. Then $ u^{*} ,v^{*} \in L^{p}(\mathbb{R}^{N}) $ and
$$ ||u^{*}||_{p}=||u||_{p},\qquad ||u^{*}-v^{*}||_{p} \leq ||u-v||_{p}.$$
\end{lemma}
\begin{lemma}
Let $ 1<p<N $ and $ u \in D^{1,p}(\mathbb{R}^{N}) $  (respectively, in $ W^{1,p}(\mathbb{R}^{N})$). Then $ u^{*} $ belongs to $ D^{1,p}(\mathbb{R}^{N})$ (respectively, in $ W^{1,p}(\mathbb{R}^{N}) $), and we have
\begin{equation*}
\int_{\mathbb{R}^{N}}|\nabla u^{*}(x)|^{p}dx\leq \int_{\mathbb{R}^{N}}|\nabla u(x)|^{p} dx.
\end{equation*}
\end{lemma}

We will be frequently using the following well-known decay and compactness properties of radial functions on $\mathbb{R}^{N}$.

\begin{lemma}[\cite{Jiabao }]\label{4.5.}
Assume that $ 1 < p < N $. Then there exists $ C=C(N,p) > 0 $ such that for all $ u\in D_{r}^{1,p}(\mathbb{R}^{N}) $,
\begin{equation}
|u(x)|\leq C|x|^{-\frac{N-p}{p}}||\nabla u||_{L^{p}(\mathbb{R}^{N})}.
\end{equation}
\end{lemma}

\begin{lemma}[Compactness of the radial embedding \cite{Jiabao }]\label{1.14}
Let $ 1<p<N $. Then we have the following continuous embedding
\begin{equation}
W_{r}^{1,p}(\mathbb{R}^{N})\hookrightarrow L^{q}(\mathbb{R}^{N})
\end{equation}
for $ p \leq q \leq   p^{*}:=\frac{pN}{N-p}  $ when $ p^{*} <\infty $ and for $ p \leq q < \infty$ when $ p^{*} =\infty $.
Furthermore, the embedding is compact for $ p < q < p^{*} $.
\end{lemma}

\begin{lemma}\label{41}
\begin{itemize}
\item[$(1)$] Let $s\geq1$ and let $u \in L^{s}(\mathbb{R}^{N})$  be a radial nonincreasing function. Then for every $x\neq0,$
\begin{equation}\label{p129}
|u(x)|\leq C |x|^{\frac{-N}{s}}||u||_{s},
\end{equation}
where $C=C(s,N)$, see e.g. \cite{Berestycki-Lions}. 
\item[$(2)$]Let $u_{n}\in D^{1,p}(\mathbb{R}^{N})$ be a sequence of radial functions such that $u_{n}\rightharpoonup u$ in $D^{1,p}(\mathbb{R}^{N})$. Then, passing if necessary to a subsequence, it holds that
$$u_{n}\rightarrow u \quad \text{in} \; L^{\infty}(\mathbb{R}^{N}\backslash B_{r}(0)) \quad \text{and}\quad\; L^{s}(\mathbb{R}^{N}\backslash B_{r}(0)) \quad \forall\; r>0, \;s>p^{*}.$$
\end{itemize}
\end{lemma}

\begin{proof}
Since $(u_n)_{n\in \mathbb N}\in D^{1,p}(\R^N)$ is a radial sequence, setting $f_n(|x|)=u_n(x)$ from the fundamental theorem of calculus and H\"older's inequality for all $|x|>|y|>r>0$ it holds that
\begin{equation*}\label{hold}
|u_n(x)-u_n(y)|\leq \int^{|x|}_{|y|} |f'_n(t)|\textrm{d}t \leq|x-y|^{1/p'} |y|^{(1-N)/p}|\mathbb S^{N-1}|^{-1/p}\|\nabla u_n\|_{L^p(\R^N)}
\end{equation*}
and as a consequence, since $u_n \rightharpoonup u$ is bounded, for all $x,y\in \R^N\setminus B_r(0)$ and a uniform constant $C>0$ we have that

\begin{equation}\label{hold}
|u_n(x)-u_n(y)|\leq  C r^{(1-N)/p}|x-y|^{1/p'}.
\end{equation}
Namely, $(u_n)_{n\in \mathbb N}$ is bounded in $C^{0,1/p'}(\R^N\setminus B_r(0))$ and by the locally compact embedding it is strongly convergent to $u$
in $L^{\infty}_{\textrm{loc}}(\R^N\setminus B_r(0)).$ This and Lemma \ref{4.5.} yield the convergence in $L^{s}(\R^N\setminus B_r(0)).$
\end{proof}

\section{Comparison principle for the $p$-Laplacian}
Let $G\subseteq\R^N$ be a domain.
We say that $ 0\leq v \in W_{loc}^{1,p}(G) $ satisfies condition $(S)$ if:
\begin{itemize}
\item[$(S)$] there exists $(\theta_{n})_{n\in \mathbb{N}}\subset W_{c}^{1,\infty}(\mathbb{R}^{N}) $ such that $0\leq \theta_{n}\rightarrow 1 \; \text{a.e.}\; \text{in}\; \mathbb{R}^{N}$ and
$$ \int\limits_{G} \mathcal{R}(\theta_{n}v,v) \rightarrow 0,\quad \text{as} \quad n\rightarrow +\infty. $$
\end{itemize}
where $\mathcal{R}$ is defined by
\begin{equation}\label{161}
\mathcal{R}(w,v):=\big |\nabla w\big|^{p}-\nabla\Big(\frac{w^{p}}{v^{p-1}}\Big)\big|\nabla v\big|^{p-2}\nabla v.
\end{equation}
Notice that if $G $ is bounded and $v\in W^{1,p}(G) $ then condition $ (S)$ is trivially satisfied with $ \theta=1$ in $ G$.
In case of an unbounded domain $G$, condition $(S)$ ensures that the subsolution $v$ is sufficiently {\em small} at infinity,
in order to respect the comparison principle (see \cite{Liskevich}).  \\
Using condition $ (S)$, we formulate a version of comparison principle for a $p$-Laplacian with a general negative potential (see e.g. \cite{ Shafrir B,Liskevich,Poliakovsky}).
\begin{theorem}[Comparison principle for $p$-Laplacian]\label{A52}
Let $ 0< u\in W_{loc}^{1,p}(G) \cap C(\bar{G})$ be a super-solution and $v\in W_{loc}^{1,p}(G) \cap C(\bar{G})$ a sub-solution to the equation
\begin{equation}\label{2.32}
-\Delta_p u+V|u|^{p-2}u=0\quad\text{in } \;G,
\end{equation}
where $V\in L_{loc}^{\infty}(G)$.
If $ G $ is an unbounded domain, assume in addition that $ \partial G \neq \emptyset$ and $v^{+} $ satisfies condition $ (S) $.
Then $u\geq v $ on $\partial G  $ implies $ u\geq v $ in $G $.
\end{theorem}

Below we prove a simple sufficient condition for assumption $(S)$ to hold. 

\begin{lemma}\label{A55}
If
$0\le v\in D_{rad}^{1,p}(\mathbb{R}^{N})$ then $v$ satisfies $(S)$.
\end{lemma}

\begin{proof}
Following \citelist{\cite{Shafrir B}\cite{Liskevich}}, define
$$\eta_{R}(r)=
\begin{cases}
	1,& 0\leq r\leq R\\
	\frac{\log \frac{R^{2}}{r}}{\log R}, & R\leq r\leq R^{2},\\
	0,& r\geq R^{2},
\end{cases}  $$
and note that $|\eta_{R}|\leq 1$ a.e. in $\mathbb{R}^{N}$ and $|\eta'_{R}|\le \frac{c}{ \log R } r^{-1} $.
We are going to show that
$$\int\limits_{\mathbb{R}^{N}} \mathcal{R}(\eta_{R}v,v) \rightarrow 0\quad \text{as} \quad R\rightarrow \infty. $$
Using the Picone's identity \cite{Dias,Allegretto} and inequalities \cite[Lemma 7.4]{Shafrir B},
it is straightforward to deduce the inequalities
\begin{eqnarray}
\mathcal R(\eta_R v ,v)&\le& c_1|v\,(\eta_R)_r'|^p,
\hspace{9em}\qquad (1<p\le 2),\label{e-sub<2}\\
\mathcal R(\eta_R v,v)&\le& c_2|\eta_R v_r'|^{p-2}|v(\eta_R)_r'|^2+c_3|v\,(\eta_R)_r'|^p,
\qquad(p>2).\label{e-sub>2}
\end{eqnarray} 

\smallskip\noindent
{\em Case $1<p\le2$.}
Using \eqref{e-sub<2} and Ni's decay estimate Lemma \ref{4.5.} on $v\in D_{rad}^{1,p}(\mathbb{R}^{N})$, 
$$v\leq c|x|^{-\frac{N-p}{p}},$$
by a direct calculation we obtain
\begin{eqnarray}
\int\limits_{\mathbb{R}^{N}} \mathcal{R}(\eta_{R}v,v)dx 
&\leq&  c_1\int\limits_{R}^{R^{2}}|v (\eta_{R})'_r |^{p} r^{N-1} dr
\le c\int\limits_{R}^{R^{2}} \Big|r^{-\frac{N-p}{p}}\frac{1}{ \log R } r^{-1}\Big|^{ p}r^{N-1} dr\nonumber\\
&\leq& 
\frac{C}{(\log R)^{p-1}}\rightarrow 0 \quad \text{as}\; R\rightarrow\infty.\label{comp-i}
\end{eqnarray}

\smallskip\noindent
{\em Case $p>2$.}
By H\"older and \eqref{comp-i} we conclude
\begin{multline*}
\int_0^{+\infty}|\eta_{R}v'|^{p-2}|v(\eta_{R})'_r)|^{2}r^{N-1}dr \le 
\left(\int_0^{+\infty}|\eta_{R}v'|^pr^{N-1}dr\right)^\frac{p-2}{p}
\left(\int_0^{+\infty}|v(\eta_{R})'_r)|^{p}r^{N-1}dr\right)^\frac{2}{p}
\\
\le c\|v\|_{D^{1,p}(\mathbb{R}^{N})}^{p-2}\left(\int_R^{R^2}|v(\eta_{R})'_r)|^{p}r^{N-1}dr\right)^\frac{2}{p}
\rightarrow 0 \quad \text{as} \; R\rightarrow \infty.
\end{multline*}
Taking into account \eqref{e-sub>2} and once again \eqref{comp-i}, the conclusion follows.
\end{proof}

\begin{remark}
While the statement of Lemma \ref{A55} is sufficient for our purposes, it is far from optimal.
See \cite[Appendix B]{Liskevich} for constructions of radial functions $v\not\in D^{1,p}(\mathbb{R}^{N})$ which satisfy assumption  $(S)$.
\end{remark}




\newpage
\begin{thebibliography}{99}		
		
						
		\bib{Allegretto} {article} {
			AUTHOR = {Allegretto, Walter}
			AUTHOR = { Huang, Yin Xi},
			TITLE = {A {P}icone's identity for the {$p$}-{L}aplacian and
				applications},
			JOURNAL = {Nonlinear Anal.},
			FJOURNAL = {Nonlinear Analysis. Theory, Methods \& Applications. An
				International Multidisciplinary Journal. Series A: Theory and
				Methods},
			VOLUME = {32},
			YEAR = {1998},
			NUMBER = {7},
			PAGES = {819--830},
		}
		
		
		\bib{Ambrosetti} {book} {
			TITLE = {Perturbation methods and semilinear elliptic problems on $\mathbb{R}^{n}$ },
			AUTHOR = {Ambrosetti, Antonio},
			AUTHOR = { Malchiodi, Andrea},
			SERIES = {Progress in Mathematics},
			VOLUME = {240},
			PUBLISHER = {Birkh\"auser Verlag, Basel},
			YEAR = {2006},
			PAGES = {xii+183},
			ISBN = {978-3-7643-7321-4; 3-7643-7321-0},
		}
		

	

	
		
		
		
		
		
		
		
		
		
		
		
		
		
		\bib{Berestycki-Lions}{article}{
			author={Berestycki, H.},
			author={Lions, P.-L.},
			title={Nonlinear scalar field equations. I. Existence of a ground state},
			journal={Arch. Rational Mech. Anal.},
			volume={82},
			date={1983},
			number={4},
			pages={313--345},
		}
		
		
		
		
		
		\bib{Brezis}{article}{
			AUTHOR = {Br{\'e}zis, Ha{\"{\i}}m}  AUTHOR = {Lieb, Elliott},
			TITLE = {A relation between pointwise convergence of functions and
				convergence of functionals},
			JOURNAL = {Proc. Amer. Math. Soc.},
			FJOURNAL = {Proceedings of the American Mathematical Society},
			VOLUME = {88},
			YEAR = {1983},
			NUMBER = {3},
			PAGES = {486--490},
		}
		
		
		
		
		
		
		
		
		\bib{Caffarelli} {article} {
			AUTHOR = {Caffarelli, Luis A.}
			AUTHOR = { Gidas, Basilis }
			AUTHOR = {  Spruck, Joel},
			TITLE = {Asymptotic symmetry and local behavior of semilinear elliptic
				equations with critical {S}obolev growth},
			JOURNAL = {Comm. Pure Appl. Math.},
			FJOURNAL = {Communications on Pure and Applied Mathematics},
			VOLUME = {42},
			YEAR = {1989},
			NUMBER = {3},
			PAGES = {271--297},
		}
		
		
		\bib{Cahn} {article} {
			title={Free energy of a nonuniform system. III. nucleation in a $2$-Component incompressible fluid},
			author={Cahn, John Wi }, author={Hilliard, John Ei},
			journal={J. Chem. phys.},
			volume={31},
			number={3},
			pages={688--699},
			year={1959},
			publisher={AIP Publishing}
		}
		
		
		\bib{Coleman} {article} {	
			title={Fate of the false vacuum: Semiclassical theory},
			author={Coleman, Sidney},
			journal={Physical Review D},
			volume={15},
			number={10},
			pages={2929},
			year={1977},
			publisher={APS}
			
		}
		
		
		\bib{Cross} {article} {
			title={Pattern formation outside of equilibrium},
			author={Cross, Mark Ch}, author={Hohenberg, Pierre Ch},
			journal={Reviews of modern physics},
			volume={65},
			number={3},
			pages={851-1112},
			year={1993},
			publisher={APS}	
		}
		
		
		
		
		
		
		
		\bib{Degiovanni} {article} {
			AUTHOR = {Degiovanni, Marco}
			AUTHOR = {Musesti, Alessandro}
			AUTHOR = {Squassina,	Marco},
			TITLE = {On the regularity of solutions in the {P}ucci-{S}errin
				identity},
			JOURNAL = {Calc. Var. Partial Differential Equations},
			FJOURNAL = {Calculus of Variations and Partial Differential Equations},
			VOLUME = {18},
			YEAR = {2003},
			NUMBER = {3},
			PAGES = {317--334},
		}
		
		
		
		
		
		
		
		\bib{Dias} {article} {
			AUTHOR = {D{\'{\i}}az, Jes{\'u}s Ildefonso }
			AUTHOR = {Sa{\'a}, Jos{\'e}	Evaristo},
			TITLE = {Existence et unicit\'e de solutions positives pour certaines
				\'equations elliptiques quasilin\'eaires},
			JOURNAL = {C. R. Acad. Sci. Paris S\'er. I Math.},
			FJOURNAL = {Comptes Rendus des S\'eances de l'Acad\'emie des Sciences.
				S\'erie I. Math\'ematique},
			VOLUME = {305},
			YEAR = {1987},
			NUMBER = {12},
			PAGES = {521--524},
			ISSN = {0249-6291},
			CODEN = {CASMEI},
			MRCLASS = {35J60 (35B05)},
			MRNUMBER = {916325},
			MRREVIEWER = {J.-P. Gossez},
		}
		
		
		\bib{DiBenedetto} {article} {
			AUTHOR = {DiBenedetto, E.},
			TITLE = {{$C^{1+\alpha }$} local regularity of weak solutions of
				degenerate elliptic equations},
			JOURNAL = {Nonlinear Anal.},
			FJOURNAL = {Nonlinear Analysis. Theory, Methods \& Applications. An
				International Multidisciplinary Journal. Series A: Theory and
				Methods},
			VOLUME = {7},
			YEAR = {1983},
			NUMBER = {8},
			PAGES = {827--850},
		}
		
		
	
		
		\bib{FaMeWi} {article} {
			AUTHOR = {Farina, Alberto},
			AUTHOR = {Mercuri, Carlo},
			AUTHOR = {Willem, Michel},
						TITLE = {A Liouville theorem for the
				{$p$}-{L}aplacian and related questions},
			JOURNAL = {arXiv:1711.11552v2},
			YEAR = {2017},
		}		
		
		
		\bib{Pinchover} {article} {
			AUTHOR = {Fraas, Martin},
			AUTHOR = {Pinchover, Yehuda},
			TITLE = {Isolated singularities of positive solutions of
				{$p$}-{L}aplacian type equations in {$\Bbb{R}^d$}},
			JOURNAL = {J. Differential Equations},
			FJOURNAL = {Journal of Differential Equations},
			VOLUME = {254},
			YEAR = {2013},
			NUMBER = {3},
			PAGES = {1097--1119},
		}
		
		
		
		
		
		
	
		
		
		
		
		
		\bib{Gazzola} {article} {
			AUTHOR = {Gazzola, Filippo}
			AUTHOR = {Serrin, James},
			TITLE = {Asymptotic behavior of ground states of quasilinear elliptic
				problems with two vanishing parameters},
			JOURNAL = {Ann. Inst. H. Poincar\'e Anal. Non Lin\'eaire},
			FJOURNAL = {Annales de l'Institut Henri Poincar\'e. Analyse Non
				Lin\'eaire},
			VOLUME = {19},
			YEAR = {2002},
			NUMBER = {4},
			PAGES = {477--504},
		}
		
		\bib{Gazzola-Serrin-Tang} {article} {
			AUTHOR = {Gazzola, Filippo},
			AUTHOR = {  Serrin, James}, AUTHOR = {Tang, Moxun},
			TITLE = {Existence of ground states and free boundary problems for
				quasilinear elliptic operators},
			JOURNAL = {Adv. Differential Equations},
			FJOURNAL = {Advances in Differential Equations},
			VOLUME = {5},
			YEAR = {2000},
			NUMBER = {1-3},
			PAGES = {1--30},
		}
		
		
		
						
					
		\bib{Guedda} {article} {
			AUTHOR = {Guedda, Mohammed}
			AUTHOR = {V{\'e}ron, Laurent},
			TITLE = {Local and global properties of solutions of quasilinear
				elliptic equations},
			JOURNAL = {J. Differential Equations},
			FJOURNAL = {Journal of Differential Equations},
			VOLUME = {76},
			YEAR = {1988},
			NUMBER = {1},
			PAGES = {159--189},
		}
		
		\bib{Guedda2} {article} {
			AUTHOR = {Guedda, Mohammed}
			AUTHOR = {V{\'e}ron, Laurent},
			TITLE = {Quasilinear elliptic equations involving critical Sobolev exponents},
			JOURNAL = {Nonlinear Anal.},
			FJOURNAL = {Nonlinear Analysis},
			VOLUME = {13},
			YEAR = {1989},
			NUMBER = {1},
			PAGES = {879--902},
				
			}
		
		

		
		
		
		
		
		
		
		
		
		
		
		
		
		
		
		
		
		
		
		
		
		\bib{Lion} {article} {
			AUTHOR = {Lions, P.-L.},
			TITLE = {The concentration-compactness principle in the calculus of
				variations. {T}he limit case. {I}},
			JOURNAL = {Rev. Mat. Iberoamericana},
			FJOURNAL = {Revista Matem\'atica Iberoamericana},
			VOLUME = {1},
			YEAR = {1985},
			NUMBER = {1},
			PAGES = {145--201},
		}
		\bib{Liskevich} {article}{
			AUTHOR = {Liskevich, Vitali}
			AUTHOR =  {Lyakhova, Sofya}
			AUTHOR =  {Moroz, Vitaly},
			TITLE = {Positive solutions to nonlinear {$p$}-{L}aplace equations with
				{H}ardy potential in exterior domains},
			JOURNAL = {J. Differential Equations},
			FJOURNAL = {Journal of Differential Equations},
			VOLUME = {232},
			YEAR = {2007},
			NUMBER = {1},
			PAGES = {212--252},
		}
		
		
		
		
		
		
		
		\bib{MercuriB} {article} {
			AUTHOR = {Mercuri, Carlo}
			AUTHOR = {Riey, Giuseppe}
			AUTHOR = {Sciunzi, Berardino},
			TITLE = {A regularity result for the {$p$}-{L}aplacian near uniform
				ellipticity},
			JOURNAL = {SIAM J. Math. Anal.},
			FJOURNAL = {SIAM Journal on Mathematical Analysis},
			VOLUME = {48},
			YEAR = {2016},
			NUMBER = {3},
			PAGES = {2059--2075},
		}
		
		
		
			
			\bib{MercuriSq} {article} {
			AUTHOR = {Mercuri, Carlo}
			AUTHOR = {Squassina, Marco},
			TITLE = {Global compactness for a class of quasi-linear elliptic problems},
			JOURNAL = {Manus. Math. },
			FJOURNAL = {Manuscripta Mathematica },
			VOLUME = {140},
			YEAR = {2013},
			NUMBER = {1},
			PAGES = {119--144},
					}
		
		
		\bib{Mercuri} {article} {
			AUTHOR = {Mercuri, Carlo}
			AUTHOR = {Willem, Michel},
			TITLE = {A global compactness result for the {$p$}-{L}aplacian
				involving critical nonlinearities},
			JOURNAL = {Discrete Contin. Dyn. Syst.},
			FJOURNAL = {Discrete and Continuous Dynamical Systems. Series A},
			VOLUME = {28},
			YEAR = {2010},
			NUMBER = {2},
			PAGES = {469--493},
		}
		
		
		\bib{Merle} {article} {
			AUTHOR = {Merle, F.}
			AUTHOR = {Peletier, L. A.},
			TITLE = {Asymptotic behaviour of positive solutions of elliptic
				equations with critical and supercritical growth. {II}. {T}he
				nonradial case},
			JOURNAL = {J. Funct. Anal.},
			FJOURNAL = {Journal of Functional Analysis},
			VOLUME = {105},
			YEAR = {1992},
			NUMBER = {1},
			PAGES = {1--41},
		}
		
		
		
		\bib{Moroz-Muratov}{article}{
			author={Moroz, Vitaly},
			author={Muratov, Cyrill B.},
			title={Asymptotic properties of ground states of scalar field equations
				with a vanishing parameter},
			journal={J. Eur. Math. Soc. (JEMS)},
			volume={16},
			date={2014},
			number={5},
			pages={1081--1109},
		}
		
		
		\bib{Muratov} {article} {
			
			title={Breakup of universality in the generalized spinodal nucleation theory},
			author={Muratov, Cyrill},
			author={Vanden-Eijnden, Eric},
			journal={Journal of Statistical Physics},
			volume={114},
			number={3-4},
			pages={605--623},
			year={2004},
			publisher={Springer}
			
		}
		
		
		
		
		
		
		
		
		
		
		
		
		
		
		
		
		\bib{Pohozaev }{article}{
			author = {Poho\v{z}aev, Stanislav },
			TITLE = {On the eigenfunctions of the equation {$\Delta u+\lambda
					f(u)=0$}},
			JOURNAL = {Dokl. Akad. Nauk SSSR},
			FJOURNAL = {Doklady Akademii Nauk SSSR},
			VOLUME = {165},
			YEAR = {1965},
			PAGES = {36--39},
			ISSN = {0002-3264},
			MRCLASS = {35.80},
			MRNUMBER = {0192184 (33 \#411)},
			MRREVIEWER = {O. E. Gheorghiu},
		}
		
		\bib{Poliakovsky} {article} {
			AUTHOR = {Poliakovsky, Arkady}
			AUTHOR = {Shafrir, Itai},
			TITLE = {A comparison principle for the {$p$}-{L}aplacian},
			BOOKTITLE = {Elliptic and parabolic problems ({R}olduc/{G}aeta, 2001)},
			PAGES = {243--252},
			PUBLISHER = {World Sci. Publ., River Edge, NJ},
			YEAR = {(2002)},
			MRCLASS = {35J60 (35B05)},
			MRNUMBER = {1937544},
		}
		
		
		\bib{Pucci }{article} {
			AUTHOR = {Pucci, Patrizia} AUTHOR = {Serrin, James},
			TITLE = {A general variational identity},
			JOURNAL = {Indiana Univ. Math. J.},
			FJOURNAL = {Indiana University Mathematics Journal},
			VOLUME = {35},
			YEAR = {1986},
			NUMBER = {3},
			PAGES = {681--703},
		}
		
		\bib{Pucci-Serrin} {article} {
			AUTHOR = {Pucci, Patrizia},
			AUTHOR = { Serrin, James},
			TITLE = {Uniqueness of ground states for quasilinear elliptic
				operators},
			JOURNAL = {Indiana Univ. Math. J.},
			FJOURNAL = {Indiana University Mathematics Journal},
			VOLUME = {47},
			YEAR = {1998},
			NUMBER = {2},
			PAGES = {501--528},
		}
		
		
		\bib{PucciB}{book}{
			AUTHOR = {Pucci, Patrizia},
			AUTHOR = {Serrin, James},
			TITLE = {The maximum principle},
			SERIES = {Progress in Nonlinear Differential Equations and their
				Applications, 73},
			PUBLISHER = {Birkh\"auser Verlag, Basel},
			YEAR = {2007},
			PAGES = {x+235},
			ISBN = {978-3-7643-8144-8},
			MRCLASS = {35-02 (34B15 35B50)},
			MRNUMBER = {2356201},
			MRREVIEWER = {Rodney Josu{\'e} Biezuner},
		}
		\bib{PucciC} {article} {
			AUTHOR = {Pucci, Patrizia} AUTHOR = {Servadei, Raffaella},
			TITLE = {Regularity of weak solutions of homogeneous or inhomogeneous
				quasilinear elliptic equations},
			JOURNAL = {Indiana Univ. Math. J.},
			FJOURNAL = {Indiana University Mathematics Journal},
			VOLUME = {57},
			YEAR = {2008},
			NUMBER = {7},
			PAGES = {3329--3363}
		}
		
		
		
		
		
		
		
		
		
		
		
		
		\bib{Sciunzi} {article} {
			AUTHOR = {Sciunzi, Berardino},
			TITLE = {Classification of positive {$D^{1,p}(\mathbb{R}^N)$}-solutions to the critical {$p$}-{L}aplace equation in
				{$\mathbb{R}^N$}},
			JOURNAL = {Adv. Math.},
			FJOURNAL = {Advances in Mathematics},
			VOLUME = {291},
			YEAR = {2016},
			PAGES = {12--23},
		}
		
		\bib{Tang}{article} {
			AUTHOR = {Serrin, James}
			AUTHOR = {Tang, Moxun},
			TITLE = {Uniqueness of ground states for quasilinear elliptic
				equations},
			JOURNAL = {Indiana Univ. Math. J.},
			FJOURNAL = {Indiana University Mathematics Journal},
			VOLUME = {49},
			YEAR = {2000},
			NUMBER = {3},
			PAGES = {897--923},
		}
		
		\bib{Shafrir B } {article} {
			AUTHOR = {Shafrir, Itai},
			TITLE = {Asymptotic behaviour of minimizing sequences for {H}ardy's
				inequality},
			JOURNAL = {Commun. Contemp. Math.},
			FJOURNAL = {Communications in Contemporary Mathematics},
			VOLUME = {2},
			YEAR = {2000},
			NUMBER = {2},
			PAGES = {151--189},
		}
		
		
		
		\bib{Struwe }{book}{
			author={ Struwe,Michael  },
			title={Variational Methods
				Applications to Nonlinear Partial Differential Equations and Hamiltonian Systems},
			publisher={Springer Berlin Heidelberg},
			isbn={978-3-662-02626-7},
			year={1990},
			series={},
			edition={},
			volume={},
			url={},
		}
		\bib{Jiabao }{article} {
			AUTHOR = {Su, Jiabao}  AUTHOR = {Wang, Zhi-Qiang} AUTHOR = {Willem, Michel},
			TITLE = {Weighted {S}obolev embedding with unbounded and decaying
				radial potentials},
			JOURNAL = {J. Differential Equations},
			FJOURNAL = {Journal of Differential Equations},
			VOLUME = {238},
			YEAR = {2007},
			NUMBER = {1},
			PAGES = {201--219},
		}
		
		
		
		
		\bib{Talenti }{article} {
			AUTHOR = {Talenti, Giorgio},
			TITLE = {Best constant in {S}obolev inequality},
			JOURNAL = {Ann. Mat. Pura Appl. (4)},
			FJOURNAL = {Annali di Matematica Pura ed Applicata. Serie Quarta},
			VOLUME = {110},
			YEAR = {1976},
			PAGES = {353--372},
			ISSN = {0003-4622},
			MRCLASS = {46E35},
			MRNUMBER = {0463908 (57 \#3846)},
			MRREVIEWER = {L. Cattabriga},
		}
		
		\bib{Tang Moxun} {article} {
			AUTHOR = {Tang, Moxun},
			TITLE = {Existence and uniqueness of fast decay entire solutions of
				quasilinear elliptic equations},
			JOURNAL = {J. Differential Equations},
			FJOURNAL = {Journal of Differential Equations},
			VOLUME = {164},
			YEAR = {2000},
			NUMBER = {1},
			PAGES = {155--179},
		}
		
		\bib{Tol} {article} {
			AUTHOR = {Tolksdorf, Peter},
			TITLE = {Regularity for a more general class of quasilinear elliptic equations},
			JOURNAL = {J. Differential Equations},
			FJOURNAL = {Journal of Differential Equations},
			VOLUME = {51},
			YEAR = {1984},
		
			PAGES = {126--150},
		}		
				
		
		
		
		
		
		
		
		
		
		
		
		
		\bib{van} {article} {
			
			
			TITLE = {Fronts, pulses, sources and sinks in generalized complex	
				{G}inzburg-{L}andau equations},
			AUTHOR = {van Saarloos, Wim},
			AUTHOR = {Hohenberg, P. C.},
			
			JOURNAL = {Phys. D},
			
			FJOURNAL = {Physica D. Nonlinear Phenomena},
			
			VOLUME = {56},
			
			YEAR = {1992},
			
			NUMBER = {4},
			
			PAGES = {303--367},
			
			ISSN = {0167-2789},
			
			MRCLASS = {35Q55 (58F39)},
			
			MRNUMBER = {1169610},
			
			MRREVIEWER = {J. J. C. Nimmo},
			
			
			
		}
		
		\bib{Vazquez} {article} {
			AUTHOR = {V{\'a}zquez, J. L.},
			TITLE = {A strong maximum principle for some quasilinear elliptic
				equations},
			JOURNAL = {Appl. Math. Optim.},
			FJOURNAL = {Applied Mathematics and Optimization. An International Journal
				with Applications to Stochastics},
			VOLUME = {12},
			YEAR = {1984},
			NUMBER = {3},
			PAGES = {191--202},
		}
		
		
		\bib{Vet} {article} {
			AUTHOR = {V\'etois, Jerome},
			TITLE = {A priori estimates and application to the symmetry of solutions for critical p-Laplace equations},
			JOURNAL = {  J. Differential Equations},
			FJOURNAL = {  J. Differential Equations},
			VOLUME = {260},
			YEAR = {2016},
			
			PAGES = {149--161},
		}		
		


		
		\bib{Willem b}{book}{
			title={Functional Analysis: Fundamentals and Applications},
			author={Willem, M.},
			isbn={9781461470045},
			lccn={2013936793},
			series={Cornerstones},
			url={https://books.google.co.uk/books?id=KgK9BAAAQBAJ},
			year={2013},
			publisher={Springer New York}
		}
		
		
		
		
		
		
		\bib{Unger} {article} {
			title={Nucleation theory near the classical spinodal},
			author={Unger, Chris},  author={Klein, William},
			journal={Physical Review B},
			volume={29},
			number={5},
			pages={2698-2708},
			year={1984},
			publisher={APS}	
		}
		

		
		
		
		
\end{thebibliography}
\end{document}